\documentclass{article}
\usepackage[utf8]{inputenc}

\usepackage{geometry}  
\usepackage{amssymb}
\usepackage{amsmath}
\usepackage{amsthm }
\usepackage{enumerate}
\usepackage{graphicx, wrapfig}
\usepackage{fullpage}
\usepackage{caption}
\usepackage{subcaption}
\usepackage{lipsum}
\usepackage{algorithm}
\usepackage{algorithmic}
\usepackage[hidelinks]{hyperref}
\usepackage{multicol}
\usepackage{xcolor}
\usepackage{dsfont}
\usepackage{bbm}
\usepackage{comment}

\title{Long-time behaviour of an advection-selection equation
}

\author{Jules Guilberteau%<-this % stops a space
\thanks{Sorbonne Université, CNRS, Université Paris Cité, Inria, Laboratoire Jacques-Louis Lions (LJLL), F-75005 Paris, France. {\tt\small jules.guilberteau@sorbonne-universite.fr}}, 
Camille Pouchol% <-this % stops a space
\thanks{Universit\'e Paris Cité, FP2M, CNRS FR 2036, MAP5 UMR 8145, F-75006 Paris,
France. 
         {\tt\small camille.pouchol@u-paris.fr}}
\ and 
Nastassia Pouradier Duteil% <-this % stops a space
\thanks{ 
       Sorbonne Université, Inria, Université Paris Cité, CNRS, Laboratoire Jacques-Louis Lions (LJLL), F-75005 Paris, France  {\tt\small nastassia.pouradier\_duteil@sorbonne-universite.fr}}
}

\date{}

\newtheorem{theorem}{Theorem}
\newtheorem{prop}[theorem]{Proposition}
\newtheorem{cor}[theorem]{Corollary}

\newtheorem{lem}{Lemma}

\theoremstyle{definition}
\newtheorem{definition}{Definition}

\theoremstyle{definition}
\newtheorem*{rem}{Remark} 

\newtheorem*{notation}{Notation}

\newcommand{\R}{\mathbb{R}}
\newcommand{\N}{\mathbb{N}}

\newcommand{\bn}{\bar n}
\newcommand{\e}{\varepsilon}
\newcommand{\tr}{\tilde{r}}

\newcommand{\ind}{\mathbbm{1}}

\newcommand{\dive}{\nabla \cdot }
\newcommand{\E}{\mathcal{E}}

\newcommand{\ui}{\underset{t\to +\infty}}
\newcommand{\ul}{\underset{t\to +\infty}{\longrightarrow}}

\newcommand{\co}{\mathcal{C}}

\renewcommand{\O}{\mathcal{O}}
\newcommand{\U}{\mathcal{U}}
\newcommand{\supp}{\mathrm{supp}}

\renewcommand{\bar}{\overline}

\begin{document}

\maketitle

%\tableofcontents

%\newpage

\begin{abstract}

We study the long-time behaviour of the advection-selection equation 
$$\partial_tn(t,x)+\nabla \cdot \left(f(x)n(t,x)\right)=\left(r(x)-\rho(t)\right)n(t,x),\quad \rho(t)=\int_{\R^d}{n(t,x)dx}\quad t\geq 0, \;  x\in \R^d,$$
with an initial condition $n(0, \cdot)=n^0$. 
In the field of adaptive dynamics, this equation typically describes the evolution of a phenotype-structured population over time. In this case, $x\mapsto n(t,x)$ represents the density of the population characterised by a phenotypic trait $x$, the advection term `$\nabla \cdot \left(f(x)n(t,x)\right)$' a cell differentiation phenomenon driving the individuals toward specific regions, and the selection term `$\left(r(x)-\rho(t)\right)n(t,x)$' the growth of the population, which is of logistic type through the total population size $\rho(t)=\int_{\R^d}{n(t,x)dx}$. 

%The phenomena of advection and growth have antagonistic effects. Indeed, it is well-known that the advection term drives the solution to the asymptotically stable equilibria of the corresponding ODE, whereas the growth term pushes it towards the regions where the growth function is maximised.
In the one-dimensional case $x\in \R$, we prove that the solution to this equation can either converge to a weighted Dirac mass or to a function in $L^1$.  Depending on the parameters $n^0$, $f$ and $r$, we determine which of these two regimes of convergence occurs, and we specify the weight and the point where the Dirac mass is supported, or the expression of the $L^1$-function which is reached. 

\end{abstract}

\section{Introduction}

\subsection{Advection-selection equation}

We consider the asymptotic behaviour of the advection-selection equation
\begin{align}
\begin{cases}
\partial_tn(t,x)+\nabla \cdot \left(f(x)n(t,x)\right)=\left(r(x)-\rho(t)\right)n(t,x),\quad  t\geq 0, \;  x\in \R^d\\
\rho(t)=\int_{\R^d}{n(t,x)dx},\quad  t\geq 0\\
n(0, x)=n^0(x),\quad  x\in \R^d.
\end{cases}
\label{eq intro}
\end{align}
This type of model typically comes up in the field of \textit{adaptive dynamics}. The aim is to understand how, among heterogeneous populations of individuals structured by a so-called continuous \textit{trait} or \textit{phenotype} $x$, the distribution of the density $x\mapsto n(t,x)$ evolves over time, and which phenotypes prevail in large times $t \rightarrow +\infty$.

In the model above~\eqref{eq intro}, the partial differential equation (PDE) takes into account 
\begin{itemize}
\item advection via the term $\nabla \cdot \left(f(x)n(t,x)\right)$, whereby individuals follow the flow associated with $f$,
\item growth via the term $(r(x)-\rho(t))n(t,x)$, which is of logistic type through the total population size $\rho(t) = \int_{\R^d} n(t,x) \,dx$.
\end{itemize}
%
%which typically models a phenotype-structured population of cells, and where the advection term `$\nabla \cdot \left(f(x)n(t,x)\right)$', represents a cell-differentiation phenomenon, and the term `$(r(x)-\rho(t))n(t,x)$' models selection and growth of the population. 

The literature concerning so-called phenotype-structured partial differential equations for adaptive dynamics is abundant~\cite{alfaro2017effect, bouin2012invasion, barles2009concentration, calsina2013asymptotics, calsina2005stationary,   chisholm2016effects, coville2013convergence, diekmann2005dynamics,  lorenzi2022invasion, lorz2013populational, perthame2006transport, perthame2008dirac}. These models usually take into account selection, which favors individuals with the most adapted traits in terms of growth, and mutations, which induce a slight phenotypic change upon reproduction. Mutation is often assumed to be rare and small compared to selection, \cite{dieckmann1996dynamical, geritz1998evolutionarily, metz1995adaptive}.  Models with no mutation at all have also been the subject of several studies~\cite{ almeida2019evolution, desvillettes2008selection, gyllenberg2005impossibility, jabin2011selection, lorenzi2020asymptotic, pouchol2018global}.

One way to analyse how the population adapts is to study the long-time behaviour for solutions of such PDE models. In particular, determining if the population becomes monomorphic (\textit{i.e.} the solution concentrates around a certain trait, called \textit{Evolutionary
Stable Strategy} (ESS) \cite{hines1987evolutionary}), or if phenotypic diversity is preserved is a fundamental question when studying such models. Broadly speaking, it has been shown that selection leads to concentration (around a finite number of phenotypic traits), while mutations, on the contrary, tend to regularise solutions, and, possibly, their limits \cite{bonnefon2015concentration, gyllenberg2005impossibility}. 
%Perhaps surprisingly, the studied model~\eqref{eq intro} will feature convergence to smooth functions even in the absence of terms modelling mutations. 

However, less emphasis has been put on studying the effect of advection, except for the recent few examples~\cite{chisholm2016effects, lorenzi2015dissecting, chisholm2015emergence} where most results are of numerical nature, or assume a very specific form of the functions $r$ and $f$. 

Yet, considering advection is relevant in various contexts. From the phenomenological point of view, it may represent how the environment drives the individuals towards specific regions, as opposed to more random mutations.  
It is also the rigorous way to model phenotype changes that are intrinsic to the individual, mediated by an ordinary differential equation (ODE) of the form
\begin{align}
 \dot x(t)=f(x(t)),
\label{ode advection}
\end{align}
where $x(t)\in \R^d$ denotes the phenotypic trait of the individual at time $t\geq 0$. As is well known, the PDE for the density of individuals corresponding to the sole model~\eqref{ode advection} is indeed the advection equation $\partial_tn(t,x)+\nabla \cdot \left(f(x)n(t,x)\right)=0$.
Our original motivation is that of cell differentiation, for which very refined ODE models have been developed in systems biology (see for instance~\cite{guantes2008multistable, thomas2002laws, tyson2020dynamical, zhang2014tgf}). 

The goal of the present article is to investigate the combined effect of selection and advection, assuming that mutations are absent or sufficiently small to be neglected. 
We hence study the long-time behaviour of the PDE~\eqref{eq intro}, where $n^0$ is the initial population distribution, and $\rho(t)$ is the size of the population at time $t\geq 0$. The equation incorporates advection with the flow $f$ of the corresponding ODE, and selection (or growth) through the non-linear and non-local  term $(r(x)-\rho(t))n(t,x)$. Here, $r(x)-\rho(t)$ can be interpreted as the fitness of individuals with trait $x$ inside the environment created by the total population, where the individuals are in a blind competition with all the other ones, regardless of their phenotype. %We have chosen here to neglect mutations to have a clear understanding of how differentiation and selection interact in the absence of other phenomena. 
We note that such models can rigorously be derived from stochastic individual based-models, in the limit of large populations~\cite{champagnat2006unifying, champagnat2008individual}. 

%{\color{red} Due to the small number of parameters, and since the selection term has a specific form, this equation does not claim to accurately model a cell colony, but rather to better understand the role played by the advection term in this family of models where it is rarely considered. }

%In another hand, many ODE models derived from systems biology allows to study another major phenomenon regarding the phenotypic evolution of cells: cell differentiation. This phenomenon, , has been widely models by ODE of the generic form `

In the absence of differentiation ($f\equiv 0$), the long-time behaviour of this model has been studied in detail by Beno\^it Perthame~\cite{perthame2006transport}, Tommaso Lorenzi and Camille Pouchol~\cite{lorenzi2020asymptotic},  and it has been proved that, in general, solutions typically concentrate onto a single trait. This result is rather intuitive, since this model does not take mutations into account. Solutions of the advection equation alone are also known to converge to weighted Dirac masses located at the roots of $f$ which are asymptotically stable for the ODE \eqref{ode advection}~\cite{diperna1989ordinary}. On the contrary, when considering both selection and advection as in equation~\eqref{eq intro}, the long-time behaviour is not known, to the best of our knowledge. Intuitively, two antagonistic effects will compete:
\begin{itemize}
\item advection will push the solution towards the asymptotically stable equilibria of ODE \eqref{eq intro}. 
\item growth will push the solution towards regions where $r$ is maximised. 
\end{itemize}
%One the one hand, the selection term, as proved in \cite{lorenzi2020asymptotic}, will encourage the solution to concentrate to the points where $r$ reaches its maximum. In the other hand, the advection term prompts the solution to concentrate to the roots of $f$ which are asymptotically stable for  ODE \eqref{ode advection}. Thus, selection and differentiation may have antagonistic effect
When coupling these two phenomena, our aim is to uncover whether  the solution of \eqref{eq intro} converges to a weighted Dirac mass, or if it converges to a smooth function. We show that both phenomena can occur, depending on the parameters $n^0, f$ and $r$. Perhaps surprisingly, the model~\eqref{eq intro} features convergence to smooth functions even in the absence of terms modelling mutations. 

Determining which parameters lead to convergence to a continuous function seems rather intricate in full generality. In particular, this problem cannot be addressed with traditional entropy methods as developed in~\cite{michel2005general}, since in the absence of mutations, there is no decrease of entropy.

\subsection{Main results}

In this paper, we thus develop a different strategy allowing to reduce this problem to the study of parameter-dependent integrals, which is mainly applied to the one-dimensional case ($x\in \R$). In this case, we elucidate the asymptotic behaviour for a large class of parameter values, and we show that there exist many different subcases depending on the number of zeros of the function $f$. A general statement encompassing all our results is hence rather convoluted. In order to illustrate our main results, we here focus on a few example cases which highlight the main two parameter regimes encountered for the asymptotic behaviour of \eqref{eq intro}. 
\begin{prop}
Let us assume that the parameter functions  $f$, $n^0$ and $r$ are smooth enough, that $f$ has a unique root (that we denote $x_s$), and that $f'(x_s)<0$ (which means that $ x_s$ is asymptotically stable for ODE~\eqref{ode advection}). 
Then, $\rho$ converges to $r(x_s)$, and $n$ converges to a weighted Dirac mass at $x_s$, when $t$ goes to $+\infty$.  
\end{prop} 
Hence, in the presence of a single asymptotically stable equilibrium point for ODE~\eqref{ode advection}, the solution of PDE~\eqref{eq intro} converges to a  Dirac mass at this point. In other words, the selection term is dominated by the advection term, which determines the point in which the solution concentrates. As soon as $f$ has at least two roots, the situation is much more complex and solutions may converge to $L^1$ functions, as illustrated in Figure \ref{fig 2 regimes} and exposed in the following proposition:

\begin{prop}
Let us assume that the functions $f$, $n^0$ and $r$ are smooth enough,  that $f$ has exactly two roots (that we denote $x_u$ and $x_s$, with $ x_u<x_s$), such that $f'(x_u)>0$ and $f'(x_s)<0$, which means that the points $x_u$ and $x_s$ are respectively asymptotically unstable and asymptotically stable for the ODE~\eqref{ode advection}. Moreover, let us assume that $n^0$ has its support in $[x_u,x_s]$, and that $n^0( x_u)>0$.  Then, the following alternative holds:
\begin{itemize}
\item If $r(x_s)>r(x_u)-f'(x_u)$, $n$ converges to a weighted Dirac mass at $x_s$, and  $\rho$ converges to~$r(x_s)$. 
\item If $r(x_s)<r(x_u)-f'(x_u)$, $n$ converges to a function in $L^1(x_u, x_s)$, and $\rho$ converges to $r(x_u)-f'(x_u)$.  
\end{itemize}
\label{intro two roots}
\end{prop} 
This proposition can be interpreted as follows: since $f$ is positive on $(x_u,x_s)$, the advection term drives the solution towards $x_s$. On the other hand, since $x_u$ is an equilibrium, albeit unstable, it acts as a counterweight by controlling the speed of the transition towards $x_s$ in the neighbourhood of $x_u$. Hence, in the case where  $r(x_u)-f'(x_u)$ is large enough ($r(x_u)-f'(x_u)>r(x_s)$), the growth rate around $ x_u$ is large enough to compensate for the advection term, leading to the convergence of  $n$ to a continuous function. In the other case, the advection term is dominant, and $n$ converges to a weighted Dirac mass at $x_s$. If $n^0(x_u)=0$, the toggle value between the two regimes (\textit{i.e.} the convergence to a smooth function or to a Dirac mass) changes, depending on how $n^0$ vanishes at $x_u$, and other limit functions can be reached: the complete result is detailed in Proposition \ref{2 equilibria second}. 
The method of analysis proposed in this article allows in fact to solve this problem for any function $f$ with a finite number of roots, as detailed in Proposition \ref{multi equilibria}. The case where $f$ is equal to zero on a whole interval can also be studied with our method, as highlighted by Proposition \ref{f equiv 0 on a segment}.

\begin{figure}[h]
\centering
\begin{subfigure}{\textwidth}
\includegraphics[width=0.33\textwidth]{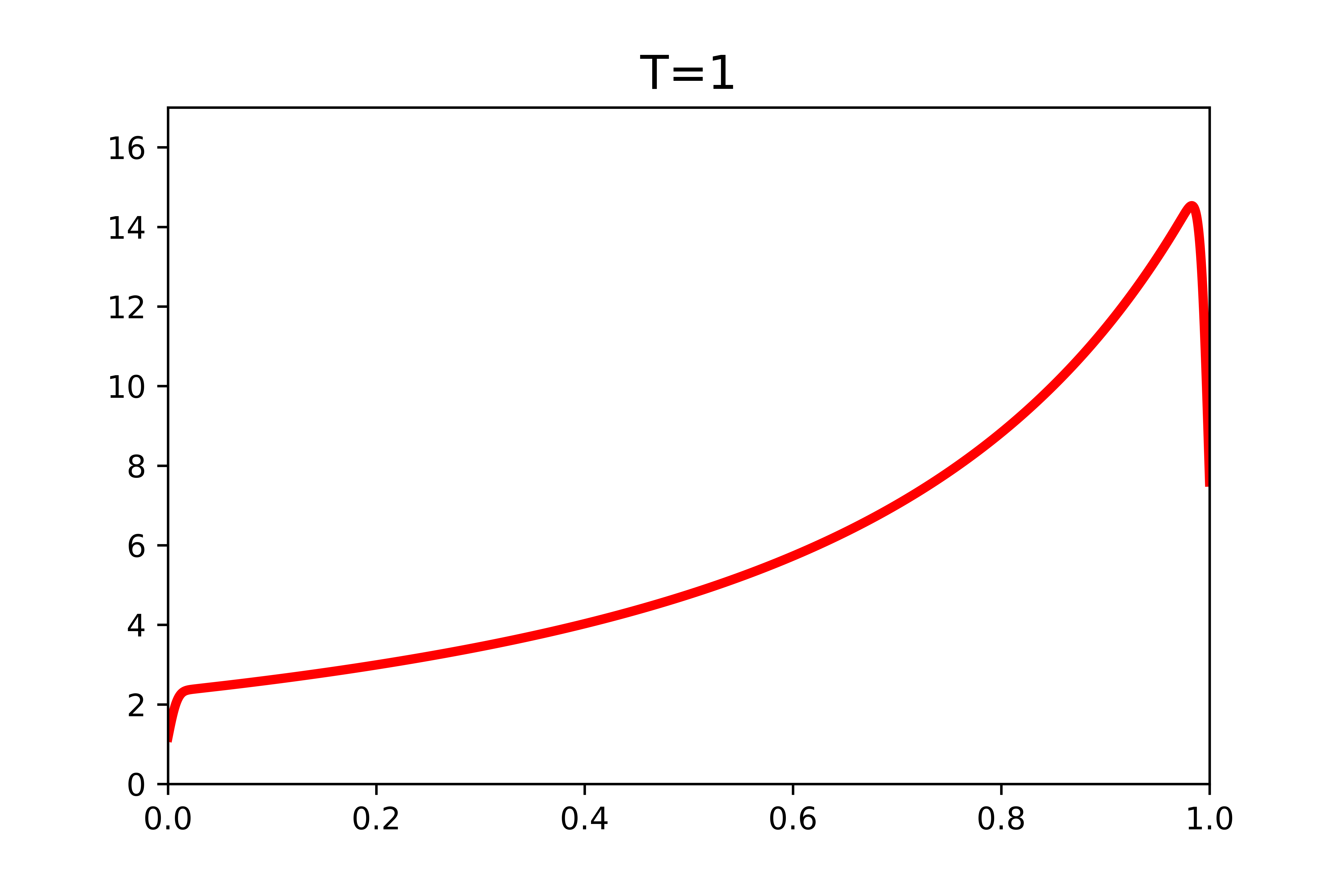}
\includegraphics[width=0.33\textwidth]{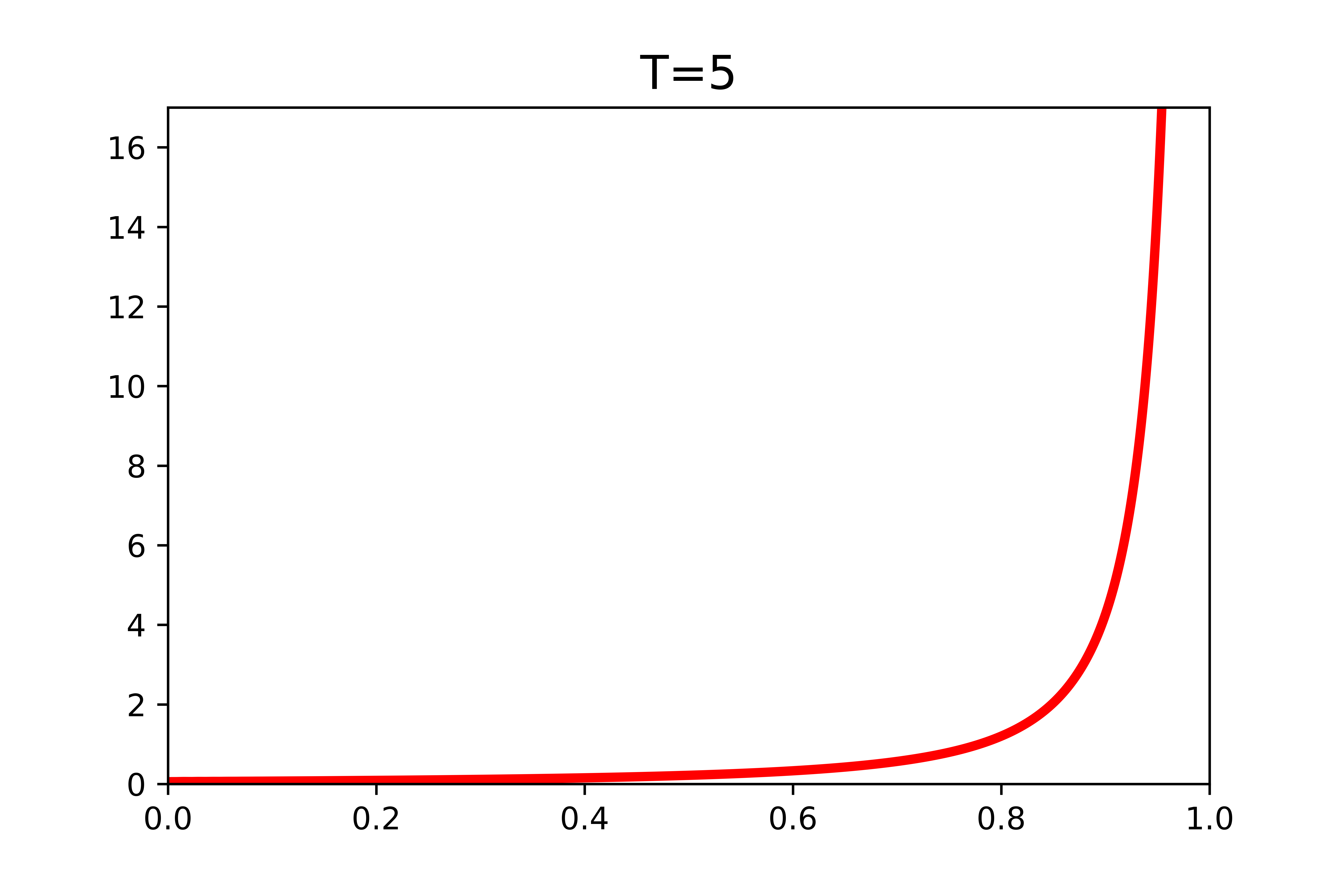}
\includegraphics[width=0.33\textwidth]{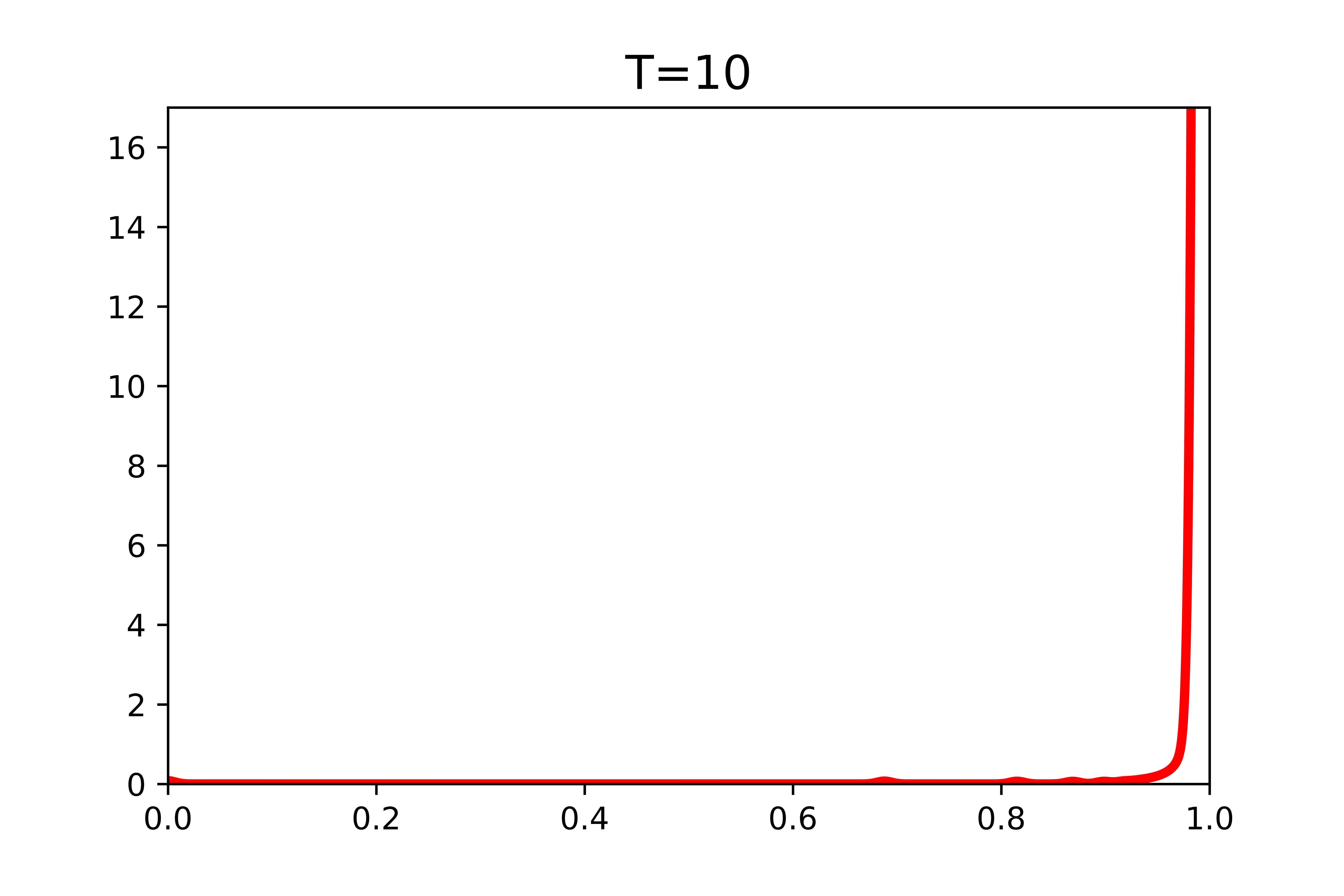}
\label{fig dirac}
\end{subfigure}
\begin{subfigure}{\textwidth}
\includegraphics[width=0.33\textwidth]{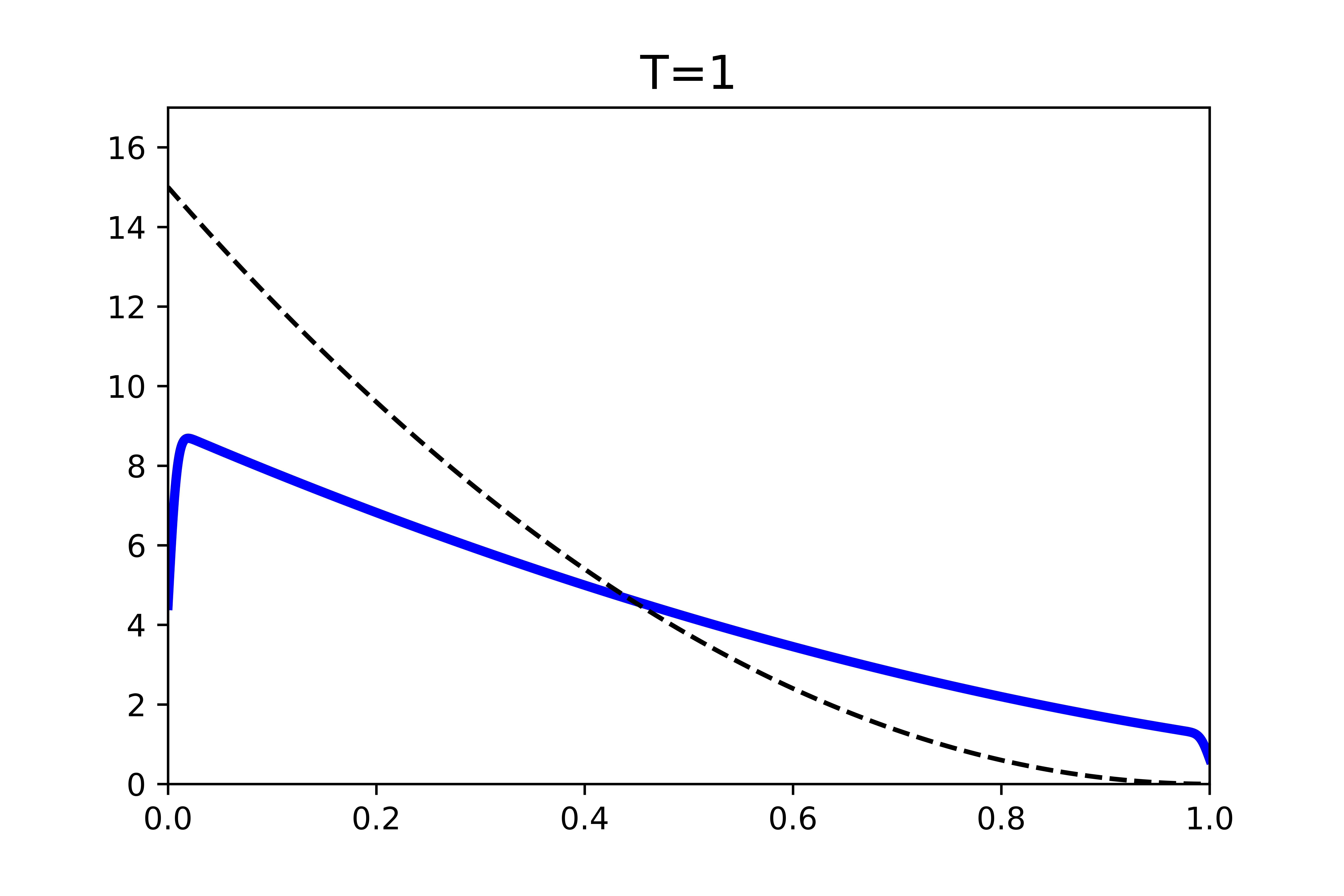}
\includegraphics[width=0.33\textwidth]{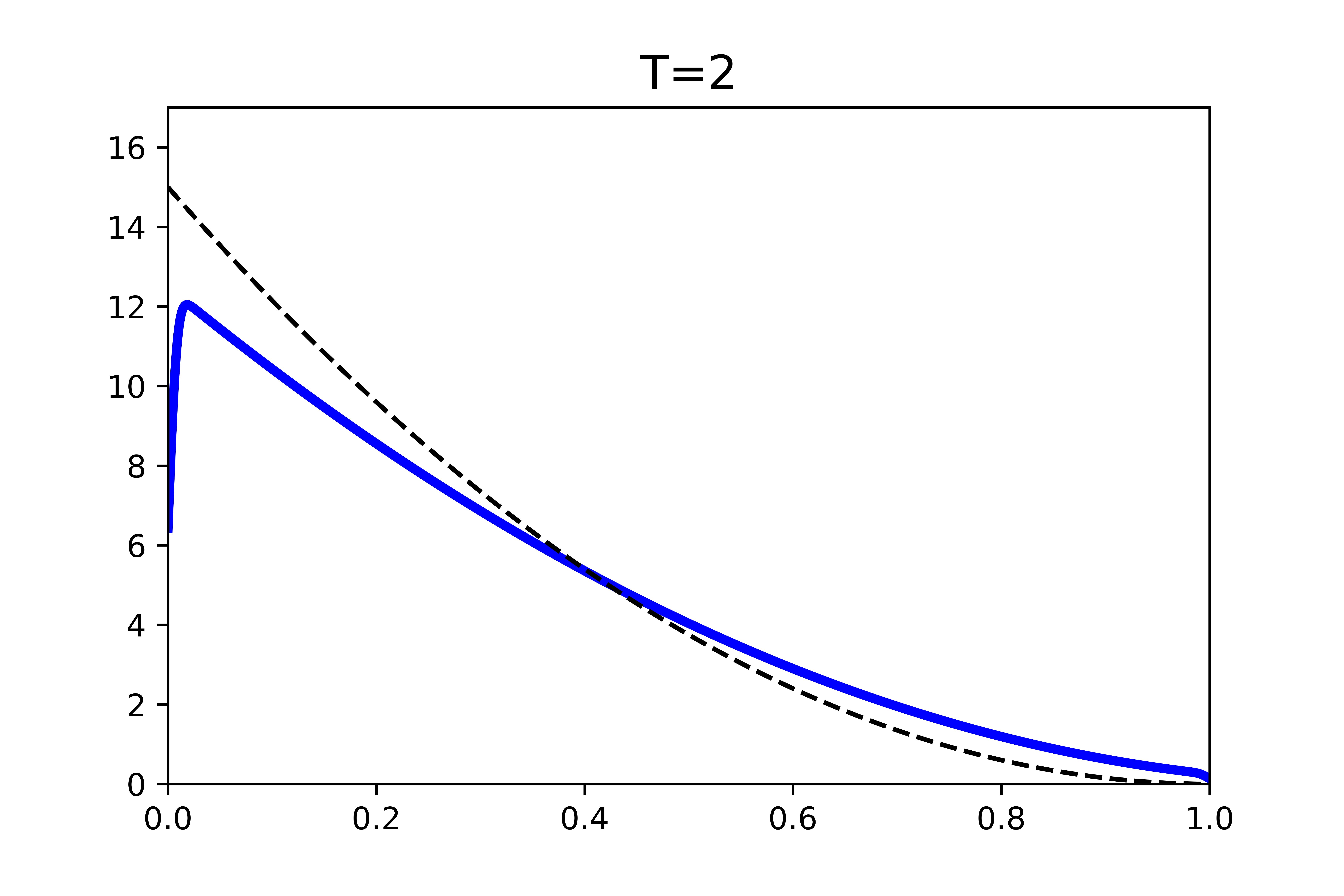}
\includegraphics[width=0.33\textwidth]{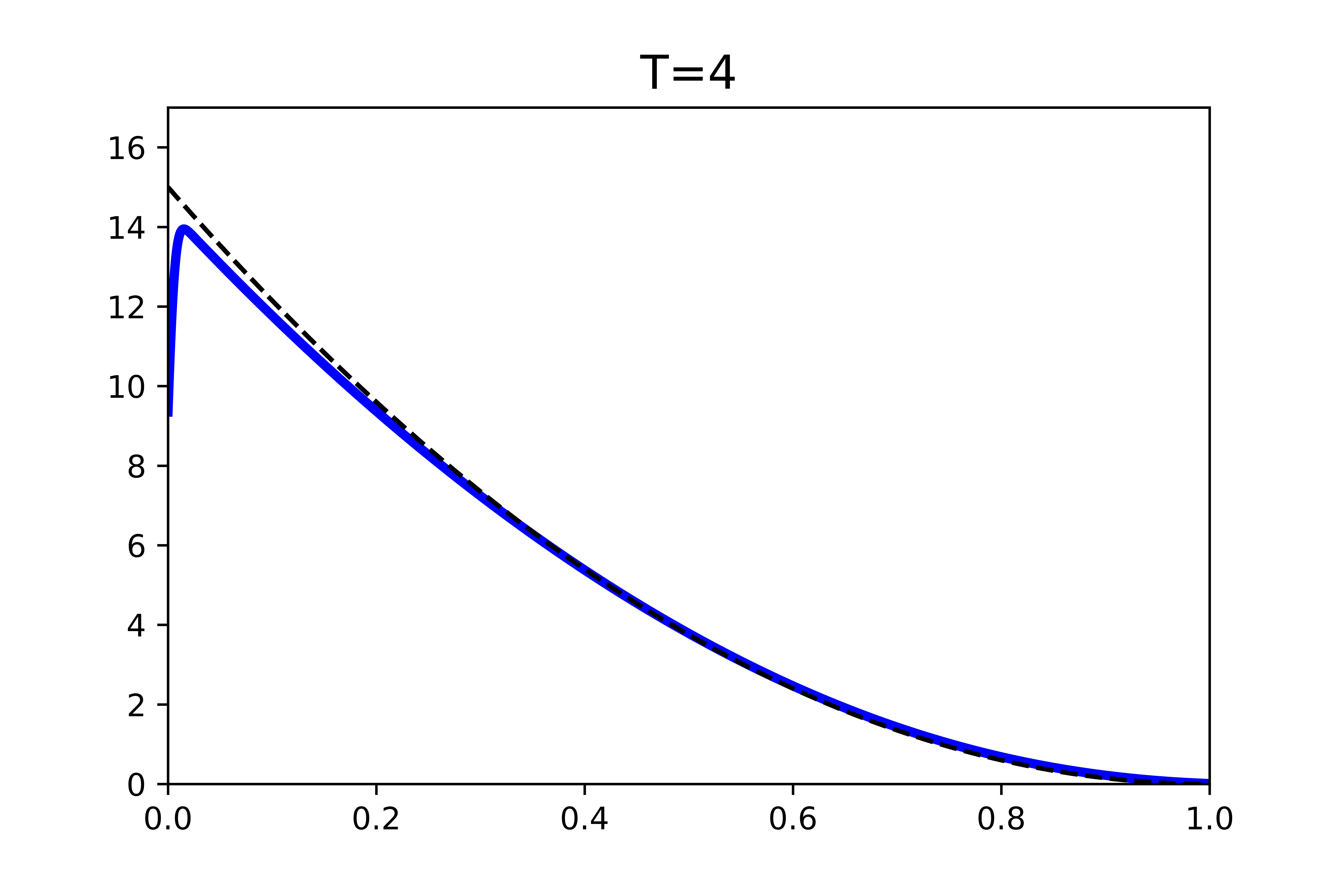}
\label{fig continuous}
\end{subfigure}

\caption{The two possible regimes of convergence stated in Proposition \ref{intro two roots}. In both cases, we have chosen $f(x)=x(1-x)$, $n^0\equiv 6$, and we work on the segment $(0,1)$ (hence $x_u=0$, $x_s=1$). The three figures above (in red) show the time evolution of the solution in the case where $r(x)=6-0.5x$ (and thus $5.5=r(1)>r(0)-f'(0)=5$), which implies, according to Proposition \ref{intro two roots}, that the solution converges to  a weighted Dirac mass at $1$. The three figures below (in blue) show the time evolution of the solution in the case where $r(x)=6-4x$, (and thus $2=r(1)<r(0)-f'(0)=5$), which implies that the solution converges to a continuous function in $L^1$. The black dashed curve represents this limit function, which can explicitly be computed (see Proposition \ref{2 equilibria second}).\\
}
\label{fig 2 regimes}
\end{figure}

\subsection{Discussion}

\paragraph{Open problems.}
Some limit cases of the problem remain unclear: we do not deal with the case of non-hyperbolic equilibria, \textit{i.e.} $\bar x\in \R$ which satisfy $f(\bar x)= f'(\bar x)=0$, and we are not able to determine what happens in the case where several carrying capacities, as defined in Section \ref{section resolution}, converge to the same maximum limit. This last case might lead to other asymptotic behaviours, such as convergence to a sum of weighted Dirac masses, or a sum of weighted Dirac masses and $L^1$-functions.
Lastly, we did not manage to elucidate the equality cases (of the form $r(x_s) =r(x_u)-f'(x_u)$).

Furthermore, even if the framework introduced in Section \ref{section resolution} could theoretically be applied in any dimension, computing the limits of the carrying capacities seems out of reach in the multidimensional case. As shown by the semi-explicit expression introduced in Subsection \ref{semi-explicit},
%\begin{align}
%n(t,x)&=n^0(Y(t,x))e^{\int_0^t{(r-\dive f)(Y(s,x))-\rho(s)ds}},
%\tag{\ref{n(t,x)}}
%\end{align}
%where $Y(\cdot, x)$ is the solution of 
the behaviour of $n$ is closely linked to that of the solutions of ODE $\dot x = f(x)$, which suggests that other asymptotic behaviours, such as convergence to a limit cycle, or chaotic behaviours (if the dimension is greater than or equal to $3$) might occur. 

These behaviours may be excluded by making specific assumptions regarding the function $f$, for example by requiring in the 2D case that ODE $\dot x = f(x)$ be competitive or cooperative. Additionally if the roots of $f$ are hyperbolic and none of them is a repellor, then $n$ cannot converge to a $L^1$-function (Proposition~\ref{non conv in L1}). Nevertheless, the question of the asymptotic limit of $n$ in this case remains open, and might be, in the presence of a saddle point, a singular measure which is not a sum of weighted Dirac masses. This situation is commonplace for some applications, since toggle switches used to model cell differentiation phenomena are usually competitive or cooperative ODE models. 

\paragraph{Perspectives.}
A natural generalisation for the model would be to model mutations, either by means of a Laplacian term or an integral term. Because of their smoothing effect, convergence to Dirac masses will typically be lost. The method developed in this paper does not seem to handle such cases well. However, it is an interesting perspective to tackle the asymptotic behaviour with entropy methods when mutations are  added~\cite{michel2005general}. 

From the numerical point of view, we have proved that the solution of this equation could be approximated with a particle method, with which we obtained the plots of Figure \ref{fig 2 regimes}. The details of the scheme, and the proof of its convergence will be published in a forthcoming article~\cite{Ernesto_Guilberteau}.  

\paragraph{Outline of the paper.} This paper is organised as follows: Section \ref{section framework} introduces the measure-theoretic framework in which convergence is considered, and includes several important reminders regarding ODE theory which will be used throughout the article. Section \ref{section resolution} details the method used to determine the asymptotic behaviour of \eqref{eq intro}, and Section \ref{section unidimensional} corresponds to a direct application of this method to several examples in the one-dimensional case. Lastly, section \ref{section higher dimension} presents two results in higher dimension which allow to determine, in some specific cases, if some initial solution can lead to a convergence to a smooth function or not.

\section{Framework and  reminders}
\label{section framework}
We consider the asymptotic behaviour of the integro-differential PDE 
\begin{equation}
\begin{cases}
\partial_tn(t,x)+\nabla \cdot (f(x)n(t,x)) = (r(x)-\rho(t))n(t,x), \quad  t\geq 0, \; x \in \R^d\\
\rho(t)=\int_{\R^d}{n(t,x)}dx, \quad t \geq 0 \\
n(0,x)=n^0(x), \quad  x\in \R^d.
\end{cases}
\tag{\ref{eq intro}}
\end{equation}
All along the article, we make the following regularity hypotheses 
\begin{itemize} 
\item $f$ is Lipschitz-continuous, and is in $\mathcal{C}^2(\R^d)$. 
\item  $r$ is positive, is in $L^1(\R^d)\cap \mathcal{C}^1(\R^d)$, and goes to zero when $\lVert x\rVert$ goes to $+\infty$. Let us note that these hypotheses imply that $r$ is bounded. 
\item  $n^0$ is in  $\mathcal{C}_c^1(\R^d)$ (the space of $\mathcal{C}^1$ functions with a compact support), is non-negative and is not the zero function.  
\end{itemize}
%These hypotheses will be made throughout the document. 
Whenever possible, we will indicate whether these hypotheses can be weakened for a given specific result. If not specified, it will be assumed that these three hypotheses hold.

From the modelling point of view, they can be justified as follows: $n^0$ denoting the initial density, it is reasonable to consider that a bounded range of phenotypic traits is initially represented; the hypothesis on $r$ at $+\infty$ is made in order to prevent an unlikely proliferation of individuals with more and more extreme  ($\lVert x \rVert \to +\infty $) phenotypic traits. 
%At least, the regularity hypothesis on each of these functions are made for \textit{ad hoc} computations and estimations, as we will see throughout the article. 

Under the above hypotheses, we can prove that there exists a unique solution $n\in \mathcal{C}\left(\R_+, L^1(\R)\right)$ for this Cauchy problem by coupling the well-known method of characteristics for the advection equation \cite{diperna1989ordinary} with the method applied in \cite{perthame2006transport} for the case $f\equiv 0$. We do not elaborate further here on the issue of existence and uniqueness, that will be addressed in a more general framework in an upcoming article \cite{Ernesto_Guilberteau}.

Since we are concerned with the long-time behaviour of the PDE~\eqref{eq intro}
%, let us start by defining the space in which convergence takes place. As 
and we expect to obtain convergence either to Dirac masses or to regular functions, the space of Radon measures is a natural setting. We start with a few usual reminders. 

\subsection{The space of Radon measures}

% We recall here the definition of the space of Radon measures and few useful results. 

%\begin{definition}[Radon measure]
%Let $\mu$ be a signed measure on $\R^d$. We say that $\mu$ is a Radon measure if $\mu$ is:
%\begin{enumerate}
%\item Inner regular, \textit{i.e } for all open set $\O\subset \R^d$, $\mu(\O)=\sup \{\mu(K), K \textrm{compact set included in } \O \}$
%\item Outer regular, \textit{i.e} for all Borel set $B\subset \R^d$, $\mu(B)=\inf \{ \mu(U), U \textrm{ open set contening } B    \}$
%\item Locally finite, \textit{i.e} for all $x\in \R^d$, there exists $V_x$ a neighbourhood of $x$ such that $\mu(V_x)<+\infty$. 
%\end{enumerate}
%\end{definition}

%We recall a common and useful result which allows to define the weak convergence in this space.

We recall that the space of \textbf{finite Radon measures} can be identified with the topological dual space of $\mathcal{C}_c(\R^d)$, \textit{i.e.} the space of continuous functions on $\R^d$ with a compact support. Thus, we say that a sequence of finite Radon measures $(\mu_k)_{k\in \N}$ \textbf{weakly converges} to a finite Radon measure $\mu$ (denoted $u_k \rightharpoonup \mu$) if 
\[ \forall \varphi\in \mathcal{C}_c(\R^d),\quad \int_{\R^d}{\varphi(x)d\mu_k(x)}\underset{k\to +\infty}{\longrightarrow} \int_{\R^d}{\varphi(x)d\mu(x)}. \]
%
%\begin{prop}
%The space of \textbf{finite} Radon measures can be identified with the topological dual space of $\mathcal{C}_c(\R^d)$. In other words, for all continuous linear form 
%$L:\mathcal{C}_c(\R^d)\to \R$, there exists a unique finite Radon measure $\mu$ such that for all $\varphi \in \mathcal{C}_c(\R^d)$, 
%\[L(\varphi)=\int_{\R^d}{\varphi(x)d\mu(x)}\]
%\end{prop}

%\begin{ex}
%\begin{enumerate}[(i)]
%\item Let $g\in L^1(\R)$. Then $\mu_g$, the natural signed measures associated with $g$, defined as 
%\[\mu_g(B)= \int_{B}{g(x)dx}\]
%is a finite Radon measure.
%\item Any dirac mass is a Radon measure.
%\end{enumerate}
%\end{ex}

%\begin{definition}[Weak convergence]
%Let $(\mu_k)_{k\in \N}$ a sequence of finite Radon measures. We say that $\mu$ weakly converges to a finite Radon measure $\mu$ (denoted $\mu_k\underset{k\to +\infty}{\rightharpoonup} \mu$) if for all $\varphi\in \mathcal{C}_c(\R^d)$
%\[\int_{\R^d}{\varphi(x)d\mu_k(x)}\underset{k\to +\infty}{\longrightarrow} \int_{\R^d}{\varphi(x)d\mu(x)}. \]

%In particular, if $u:\R_+ \times \R^d \to \R$ is an application such that, for all $t\geq 0$, $u(t, \cdot)\in L^1(\R^d)$, $u(t,\cdot)\underset{t\to +\infty}{\rightharpoonup}\mu$ if
%\[\forall \varphi\in \mathcal{C}_c(\R^d), \quad \int_{\R^d}{\varphi(x)u(t,x)dx} \ul \int_{\R^d}{\varphi(x)d\mu(x)}. \]
%\end{definition}
% All along this paper, the notion of `weak convergence', and the symbol `$\rightharpoonup$'
 %will refer to this definition. 
In this article, we will be confronted mainly with convergence to Dirac masses or to $L^1$ functions. It is clear that the convergence in $L^1$ to a certain function implies the weak convergence to this function. The following standard lemma provides a sufficient condition to prove the weak convergence to a single Dirac mass.  For completeness, we provide a proof.

\begin{lem}
Let $u:\R_+ \times \R^d \to \R$ be a non-negative mapping such that $u(t, \cdot)\in L^1(\R^d)$ for all $t\geq 0$, and $u(t, \cdot)$ is compactly supported, uniformly in $t \geq 0$. We assume that there exists $\bar x \in \R^d$, such that for all compact set $K_{\bar x}$ which does not contain  $\bar x$,  $\int_{K_{\bar x}}{ u(t,x) dx}\ul 0$, and that there exists $ V_{\bar x}$ a compact neighbourhood of $\bar x$ and $C\in \R$ such that $\int_{V_{\bar x}}{{u(t,x)dx}}\ul C$. 
Then, $u(t, \cdot)\underset{t\to +\infty}{\rightharpoonup} C\delta_{\bar x}$. 
\label{convergence Radon}
\end{lem}

\begin{proof}
 Let $\varphi\in \co_c (\R^d)$, and let $K$ be a compact set such that, for all $t \geq 0$, $\supp(u(t,\cdot))\cup  V_{\bar x}\subset K$. Then, 
\begin{align*}
\bigg\lvert \int_{\R^d}{\varphi(x)u(t,x)dx}-C\varphi(\bar x)\bigg\rvert &= \bigg\lvert \int_{K}{\varphi(x)u(t,x)dx}- \int_{K}{\varphi(\bar x)u(t,x)dx} + \int_{K}{\varphi(\bar x)u(t,x)dx} - C\varphi(\bar x) \bigg\rvert\\
&\leq  \int_{K}{  \lvert \varphi(x)-\varphi(\bar x) \rvert u(t,x) dx } + \lvert \varphi(\bar x)\rvert \bigg\lvert \int_{K}{u(t,x)dx}-C \bigg\rvert.   \\
\end{align*}
%By rewriting 
%\[\int_{K}{u(t,x)dx}= \int_{ V_{\bar x}}{u(t,x)dx}+\int_{K \backslash  V_{\bar x}}{u(t,x)dx}, \]
%ne notes that this term converges to $C$ as $t$ goes to $+\infty$, since the closure of $K\backslash V_{\bar x}$ is a compact set which does not contain $\bar x$,  and thus that 
%\[ \bigg\lvert \int_{K}{u(t,x)dx}-C \bigg\rvert  \ul 0.  \]
The second term tends to $0$ since $K$ contains $V_{\bar x}$.
It remains to prove that 
$t\mapsto \int_{\R^d}{  \lvert \varphi(x)-\varphi(\bar x) \rvert u(t,x)  dx }$
converges to zero. Let $\e>0$ be given. Since $\varphi$ is continuous, there exists $B_{\bar x}$ a neighbourhood of $\bar x$, which can be chosen as a subset of $V_{\bar x}$,  such that $\lvert \varphi(x)-\varphi(\bar x)\rvert \leq \e$, for all $x\in B_{\bar x}$. 
Thus, for all $t\geq 0$, 
\begin{align*}
\int_{K}{ \lvert \varphi(x)-\varphi(\bar x) \rvert  u(t,x)  dx } &= \int_{K\backslash B_{\bar x}}{   \lvert \varphi(x)-\varphi(\bar x) \rvert u(t,x) dx}+ \int_{B_{\bar x}}{   \lvert \varphi(x)-\varphi(\bar x) \rvert u(t,x) dx}\\
&\leq 2 \lVert \varphi \rVert_{\infty} \int_{K \backslash B_{\bar x}}{ u(t,x)  dx}+ \e \int_{B_{\bar x}}{ u(t,x) dx}. 
\end{align*}
This concludes the proof, since $t \mapsto \int_{K \backslash V_{\bar x}}{u(t,x)dx}$ converges to zero and for any $t$ large enough,  $\int_{ B_{\bar x}}{ u(t,x) dx}\leq \int_{V_{\bar x}}{u(t,x)dx}\leq C+\e$. 
\end{proof}

%\newpage

%\begin{notation}
%Let $I$ an interval of $\R$. We denote $C^1(I)$ the space of differentiable functions on $I$ with a continuous derivative on $I$, and with a continuous derivative on $I^\mathsf{o}$.

%, for any $xa<b$, the spaces $C^1((a,b))$ and $C^1([a,b])$ are drastically different. Indeed, all the functions of $C^1([a,b])$ are bounded and have a bounded derivative, which is not necessarily the case of the functions of $C^1((a,b))$. 
%\end{notation}

%Let $r\in C^1([0,1])$ a positive function, $G\in C^1(\R_+)$ a non-negative and increasing function such that $G'>0$ on $\R_+^*$, $G(0)=0$ and $G(x)$ goes to $+\infty$ as $x$ goes to $+\infty$, and let $n^0\in  L^{\infty}(0,1)$ be a non-negative function which is not identically equal to zero. 

%\section{Preliminary results and explicit expressions of solutions}

\subsection{General statement regarding the characteristics curves}
\label{section characteristic curves}
We are led to consider the \textbf{characteristics curves} associated with the advection term. In this section, we introduce some notations and state some classical results from ODE theory, that will prove to be useful later on. 

Since $f$ is assumed to be Lipschitz-continuous, the global Cauchy-Lipschitz theorem ensures the global existence on $\R_+$ and the uniqueness of the characteristic curves related to $f$ defined for all $y\in \R^d$ as the solution to the ODE 
\begin{align}
\begin{cases}
\dot X(t,y)=f(X(t,y)) \quad  t\geq 0\\
X(0,y)=y
\end{cases}.
\label{X(t,y)}
\end{align}
It is well-known that for all $t\geq 0$, $y\mapsto X(t,y)$ is a $C^1$-diffeomorphism between $\R^d$ and itself~\cite{diperna1989ordinary}, and that the inverse function of $X(t,\cdot)$, that we denote $x\mapsto Y(t,x)$, is the unique solution of 
\begin{align}
\begin{cases}
\dot Y(t,x)=-f(Y(t,x)) \quad  t\geq 0\\
Y(0,x)=x
\end{cases}. 
\label{Y(t,x)}
\end{align}
Moreover, \textbf{Liouville's formula} states that for all $t\geq 0$ and $y\in \R^d$, 
\begin{align}
\det \left(\mathrm{Jac}_yX(t,y)\right)=e^{\int_0^t \dive f(X(s,y))ds}. 
\label{det(X(t,y))}
\end{align}
It follows from the uniqueness of solutions to $\eqref{X(t,y)}$ that for all $0 \leq s \leq t$, 
\begin{align}
X(s, Y(t,x))=Y(t-s, x). 
\label{X(s,Y(t,x))}
\end{align}
%Moreover, for all $t>0$ and $s\in [0,t]$, 
%\begin{align*}
%\frac{d}{ds}X(s,Y(t,x))=f(X(s,Y(t,x))) \quad \text{and} \quad \frac{d}{ds}Y(t-s,x)=-f(Y(t-s,x)).
%\end{align*}
%Since $X(0,Y(t,x))=Y(t,x)$, it proves that, for any $0\leq s\leq t$, 
%\begin{align}
%X(s,Y(t,x))=Y(t-s,x)=X(s-t,x).
%\label{X(s,Y(t,x))}
%\end{align}
%In the same way, one can show that 
%\begin{align}
%Y(s, X(t,y))=Y(s-t,y)=X(t-s,y).
%\label{Y(s,X(t,y))}
%\end{align}

\textbf{Specific results in $\R$.}
Let us note that the behaviour of the characteristic curves is particularly simple in $\R$. Indeed, an elementary ODE analysis shows that for all $x, y\in \R$, $t\mapsto X(t, y)$ and $t\mapsto Y(t,x)$ are monotonic functions. This implies that these characteristic curves either converge to a root of $f$, or go to $\pm \infty$ as $t\to +\infty$. More precisely, if $f$ has a finite number of roots, then for all $y\in \R$ such that $f(y)>0$, $t\mapsto X(t,y)$ converges to the closest root of $f$ which is greater than $y$, or to $+\infty$ if $y$ is greater than the greatest root of $f$. Similarly, for all $y\in \R$ such that $f(y)<0$, $t\mapsto X(t,y)$ converges to the closest root of $f$ which is lesser $y$, and to $-\infty$ if $y$ is lesser the smallest root of $f$. 

%\begin{itemize}
%\item If $y$ is a root of $f$, then $t\mapsto X(t,y)$ is constant, equal to $y$. 
%\item If $f(y)>0$, and if there exists a root of $f$ greater than $y$, then $t\mapsto X(t,y)$ converges to the smallest root of $f$ greater than $y$.
%\item If $f(y)>0$, and if $y$ is greater than the greatest root of $f$, then $t\mapsto X(t,y)$ goes to $+\infty$ as $t$ goes to  $+\infty$. 
%\item If $f(y)<0$, and if there exists a root of $f$ which is less than $y$, then $t\mapsto X(t,y)$ converges to the greatest root of $f$ less than $y$. 
%\item If $f(y)<0$, and if $y$ is less than the smallest root of $f$, then $t\mapsto X(t,y)$ goes to $-\infty$ as $t$ goes to $+\infty$. 
%\end{itemize}
Moreover, if each of these roots are hyperbolic equilibrium points for the ODE $\dot x =f(x)$, \textit{i.e.} if $f'(\bar x)\neq 0$ for all $\bar x$ root of $f$, then a given root of $f$ is either asymptotically unstable (\textit{i.e.} $f'(\bar x)>0$), which implies that its basin of attraction is limited to itself, or asymptotically stable (\textit{i.e.} $f'(\bar x)<0$), which implies that its basin of attraction in an open interval containing $\overline{x}$. 

Lastly, let us recall that under these hypotheses, the convergence to an asymptotically stable point happens with an exponential speed, which means that for all $y\in \R$, $\bar x $ root of $f$, 
%\[X(t,y)\underset{t\to +\infty}{\longrightarrow} \bar x \quad \Rightarrow \quad \exists \,  C_y, \delta_y>0 : \forall t\geq 0, \;\lvert X(t,y)-\bar x \rvert\leq C_y e^{-\delta_y \, t}. \]
\[X(t,y)\underset{t\to +\infty}{\longrightarrow} \bar x \quad \Rightarrow \quad \exists \delta_y>0: \; X(t,y)-\bar{x} =\underset{t\to +\infty}{O}(e^{-\delta_y \, t}) \]
Since the reverse characteristic curves satisfy \eqref{Y(t,x)}, the same results hold for $Y(t,x)$, provided that we replace $f$ by $-f$.  In brief, the asymptotically stable equilibria become unstable for the reverse ODE, and vice versa,  and if $t\mapsto X(t,y)$ is increasing (respectively decreasing), then $t\mapsto Y(t,x)$ is decreasing (respectively increasing). 
%These basic results will be widely used, and having them in mind often provide an insight into the different results that will be stated. 

\section{Resolution method}
\label{section resolution}

The method of resolution to determine the asymptotic behaviour of $n$ that we propose here is based on the following two propositions, which are developed in the following two subsections, respectively:
\begin{enumerate}
\item For all $t\geq 0$, $x\in \R^d$, we can express $n(t,x)$ as a function which only depends on $t, x$,  on the functions $n^0$, $f$ and $r$, on the inverse characteristic curves $Y(t,x)$, and on the population size $\rho$. Therefore, knowing the limit of  $Y(t,x)$ and $\rho(t)$ as $t$ goes to $+\infty$ is enough to understand the long-time behaviour of $n$. 
\item The population size $\rho$ is the solution of a non-autonomous ODE, and its long-time behaviour may be inferred from the limit of some parameter-dependent integrals. 
\end{enumerate}

Combining these two propositions allows us to reduce the study of the asymptotic behaviour of $n$ to that of parameter-dependent integrals. 

\subsection{Semi-explicit expression of the solution}
\label{semi-explicit}
According to the definition of the characteristic curves \eqref{X(t,y)}, for all $t\geq 0$ and all $y\in \R^d$, 
\begin{align*}
\frac{d}{dt}n(t,X(t,y))=\biggl(r(X(t,y))-\dive f(X(t,y))-\rho(t)\biggr)n(t,X(t,y)),
\end{align*} 
 \textit{i.e.}
\begin{align*}
 n(t,X(t,y))=e^{\int_0^t{ \left(r(X(s,y))-\dive f (X(s,y))-\rho(s)\right)ds}}n^0(y).
\end{align*}
Replacing $y$ by $Y(t,x)$ in this last expression, we get a \textbf{semi-explicit} expression for $n$, which is expressed as a function of $t,x$ and $\rho$: 
\begin{align}
\begin{split}
n(t,x)&=n^0(Y(t,x))e^{\int_0^t{\left((r-\dive f)(X(s,Y(t,x)))-\rho(s)\right)ds}}\\
&=n^0(Y(t,x))e^{\int_0^t{\left((r-\dive f)(Y(s,x))-\rho(s)\right)ds}},
\end{split}
\label{n(t,x)}
\end{align}
The second equality holds according to equality \eqref{X(s,Y(t,x))} and the change of variable $s'=t-s$. 

Beyond the non-negativity of $n$, this semi-explicit expression shows that determining the asymptotic behaviour of $\rho$ and $Y$ is enough to uncover that of $n$. In the following section, we show that $\rho$ is the solution of a non-autonomous ODE, and that its asymptotic behaviour is related to that of parameter-dependent integrals. 

This expression also provides exhaustive information about the support of of $n(t, \cdot)$: indeed, it ensures that for all $t\geq 0$, 
\begin{align}
\supp \left(n(t, \cdot)\right)= \supp \left(n^0 \circ Y(t, \cdot ) \right)= X\left(t, \supp \left(n^0\right)\right). 
\label{supp n}
\end{align}
Since $n^0$ is assumed to have a compact support, then so does $n(t, \cdot)$ for any $t\geq 0$.

We recall that a set $\E\subset \R^d$ is said to be \textbf{positively invariant} for the ODE $\dot x = f(u)$ if for all $t\geq 0$, $X(t, \E)\subset \E$. 

With this definition in mind, it becomes clear, according to \eqref{supp n}, that if $\supp \left(n^0\right)$ is positively invariant for the ODE $\dot x=f(x)$, then $\supp \left(n(t,\cdot)\right)\subset \supp \left(n^0\right)$, for all $t\geq 0$, and, more generally, that if there exists $\E \subset \R^d$ a set which is positively invariant for this ODE such that $\supp \left(n^0 \right) \subset \E$, then $\supp \left(n(t, \cdot) \right)\subset \E$, for all $t\geq 0$. Hence, even if PDE \eqref{eq intro} is defined for all $x\in \R^d$, if the support of $n^0$ is included in a compact subset of $\R^d$ which is positively invariant, then everything happens as if we were working in this compact set. In particular, the functions $f$ and $r$ do not need to be defined outside this set. 

\subsection{ODE satisfied by the population size}

Let us start with a basic lemma which ensures that the population size $\rho$ does not blow up as $t$ tends to~$+\infty$. 
\begin{lem}[Bounds on $\rho$] 
Let $\rho$ be defined as in \eqref{eq intro}. Then for all 
$ t \geq 0$,  $\rho(t) \leq \max \left(\lVert r \rVert_{\infty},  \rho(0)\right)$. 
\end{lem}

\begin{proof}
According to \eqref{supp n}, since, $n^0$ is assumed to have a compact support, $n(t, \cdot)$ has a compact support for all $t\geq 0$. Hence, when  integrating the fist line of $\eqref{eq intro}$, the advection term vanishes, and we get
\[\dot \rho(t)=\int_{\R^d}{\big(r(x)-n(t,x)\big) n(t,x)dx}\leq \left(\lVert r \rVert_{\infty}-\rho(t)\right) \rho(t). \]
In other words, $\rho$ is a sub-solution of the logistic ODE $\dot u = \left(\lVert r \rVert_{\infty}-u\right) u$, which proves the result. \label{bounds rho}
\end{proof}

In the remainder of this section, we show that $\rho$ is in fact the solution to a non-autonomous logistic equation, which can be written in different forms.  In order to lighten the future expressions, we now denote 
\[\tr:=r-\dive f. \]
Let $\E\subset \R^d$ be any measurable subset of $\R^d$, and let us denote 
\[\rho_\E(t):=\int_\e{n(t,x)dx}, \]
which is well-defined and bounded, according to Lemma \ref{bounds rho}. 
By integrating the semi-explicit expression~\eqref{n(t,x)} of $n$ over $\E$, we obtain the equality
\begin{align}
\rho_\E(t)=S_\E(t)e^{-\int_0^t{\rho(s)ds}}, 
\end{align}
\label{rho_E}
where 
\[S_\E(t):=\int_\E{n^0(Y(t,x))e^{\int_0^t{\tr(Y(s,x))ds}}dx}\]
is a function which only depends on the parameters $f,r$ and $n^0$. This function is well-defined, and differentiable, thanks to our regularity assumptions, and since for all $t\geq 0$ $n^0(Y(t, \cdot))$ has compact support. 
Thus, under the hypothesis
that for all $t\geq 0$, $S_\E(t)>0$,
we  obtain  
\[\ln\left(\rho_\E(t)\right)=\ln\left(S_\E(t)\right)-\int_0^t{\rho(s)ds}, \] 
and finally, by differentiating and multiplying by $\rho_\E$ on both sides, 
\begin{align}
\dot \rho_\E(t)=\left(\frac{\dot S_\E(t)}{S_\E(t)}-\rho(t)\right)\rho_\E(t).
\label{dot rho E}
\end{align}
At this stage, one might be tempted to choose $\E=\R^d$ to obtain, denoting $S:= S_{\R^d}(t)$,
\begin{align}
\dot \rho(t)=\left(\frac{\dot S(t)}{S(t)}-\rho(t)\right)\rho(t).
\label{dot rho R d}
\end{align}
This proves that $\rho$ is the solution to a non-autonomous logistic equation, and the study of such equations \cite{hallam1981non} proves that if the time-dependant carrying capacity $t\mapsto \frac{\dot S(t)}{S(t)}$ converges, then $\rho$ converges to the same limit. Unfortunately, computing the limit of $t\mapsto \frac{\dot S(t)}{S(t)}$ is intricate (except in very specific cases). This brings us to introducing a more general framework, which involves simpler functions whose limit can be computed (at least in the case $x\in \R$). The idea is to partition the space $\R^d$ into several well-chosen subsets, and to consider the size of the population on each of these sets. As seen above, to obtain equations of the type~\eqref{dot rho E}, we must be cautious when choosing these subsets in order for the corresponding functions $S_\E$ to be positive. All this leads us the following proposition: 
\begin{prop}
Let $\U\subset \R^d$ be a set such that 
\begin{align}
X(\R_+\times \mathrm{supp}(n^0))\subset \U
\label{condition O}
\end{align}
and let $\left(\O_i\right)_{i\in \{1,..., N\}}$ be a finite family of open subsets of $\U$ such that 
\begin{enumerate}[(i)]
\item $\forall i \neq j, \quad \O_i \cap \O_j=\emptyset$.
\item $\nu\left( \U\backslash \bigcup\limits_{i=1}^N{\O_i}\right)=0$, where $\nu$ denotes the  Lebesgue measure.
\item  $\forall i\in \{1,... N\}, \quad \forall t\geq 0, \quad   X\big(t, \supp \left(n^0\right)\big)\cap \O_i\neq \emptyset $.
\end{enumerate}
Then, by denoting for all $i\in \{1,..., N\}$ 
\begin{align}
\rho_i(t)&:=\int_{\O_i}{n(t,x)dx},  \label{rho_i} \\
S_i(t)&:=\int_{\O_i}{n^0(Y(t,x))e^{\int_0^t{\tr(Y(s,x))ds}}dx}, \label{Si} \\
R_i(t)&:=\frac{\dot{S}_i(t)}{S_i(t)}, \label{R_i}
\end{align}
the following equation holds:
\begin{align}
\begin{cases}
\dot \rho_i(t)=\left(R_i(t)-\rho(t)\right)\rho_i(t)\quad \forall t\geq 0, \quad  \forall i\in \{1,..., N\}\\
\rho(t)=\sum\limits_{i=0}^N{\rho_i(t)}\quad \forall t\geq 0\\
\rho_i(0)>0 \quad \forall i\in \{1,..., N \}\\
\end{cases}.
\label{main ODE}
\end{align}
\label{prop Oi}
\end{prop}

\begin{rem}
Note that a sufficient condition for the third condition (iii) to hold is the following: for any $i\in \{1, ...,N\}$, there exists $x_i$ in the closure of $\O_i$ such that $f(x_i)=0$ and $n^0(x_i)>0$.
\end{rem}

\begin{proof}
As a consequence of the discussion at the beginning of this section, it is enough to prove that
\begin{enumerate}
\item For all $i\in \{1,..., N\}$ and all $t\geq 0$, \quad  $S_i(t)>0$
\item For all $t\geq 0$, \quad  $\rho(t)=\sum\limits_{i=1}^N{\rho_i(t)}$.
%\item For all $i\in \{1,..., N\}, \quad \rho_i(0)>0$. 
\end{enumerate}
%The first point is ensured by the hypothesis $X\big(t, \supp \left(n^0\right)\big)\cap \O_i\neq \emptyset$  for all $i$ and all $t$, which is equivalent to $\supp \left(n^0\right)\cap Y(t, \O_i) \neq \emptyset$, %according to the definition of $S_i$, and by the fact that $\O_i$ is an open set, which ensures, thanks to the continuity of $n^0$,  that $\{x\in \O_i : n^0(Y(t,x))>0\}$ has a positive measure for all $t\geq 0$. Since $\rho_i(0)=S_i(0)$, we also infer $\rho_i(0)>0$.

First, notice that hypothesis $(iii)$ is equivalent to
%$X\big(t, \supp \left(n^0\right)\big)\cap \O_i\neq \emptyset$, which is equivalent to 
$\supp \left(n^0\right)\cap Y(t, \O_i) \neq \emptyset$ for all $i\in \{1,..., N\}$ and all $t\geq 0$. Moreover, $\O_i$ is an open set, which ensures, thanks to the continuity of $n^0$,  that $\{x\in \O_i : n^0(Y(t,x))>0\}$ has a positive measure for all $t\geq 0$. This proves the first point by definition of $S_i$. Since $\rho_i(0)=S_i(0)$, we also infer $\rho_i(0)>0$.

The second point is due to hypothesis \eqref{condition O}: Indeed, for any $t\geq 0$, according to the semi-explicit expression of $n$ provided by \eqref{n(t,x)}, $n(t,x)= 0$ if  $Y(t,x)\notin \mathrm{supp}(n^0)$ \textit{i.e.} if $x\notin X(t, \mathrm{supp}(n^0))$, which ensures that
\[\rho(t)=\int_{\R^d}{n(t,x)dx}=\int_{\U}{n(t,x)dx}.\]
The first two hypotheses satisfied by the sets $\O_i$ ensure that $\rho(t)=\sum\limits_{i=1}^N{\rho_i(t)}$. 

\textbf{Proof of the remark:}
Let  $x_i$ be a root of $f$. A classical ODE result ensures that for all $t\geq 0$, $x\in \R^d$, $\lVert Y(t,x)-x_i\rVert\leq e^{Lt} \lVert x-x_i \rVert$, with $L>0$ the Lipschitz constant of $f$. 
Since $n^0(x_i)>0$ and $n^0$ is continuous, there exists $\e>0$ such that $B(x_i, \e)\subset \text{supp}(n^0)$. Let $t\geq 0$, $x\in \O_i\cap B(x_i, \e e^{-Lt}/2)$ (such a point does exist, by definition of the closure). Then, $Y(t,x)\in B(x_i, \e) \subset \supp(n^0)$, which ensures that $x\in X \big(t, \supp \left(n^0\right)\big)$, and thus concludes the proof. 
\end{proof}
In the one-dimensional case, assuming that $f$ has a finite number of roots, an efficient choice for the sets $\O_i$ is to take the segments between the roots of $f$ which interseect the support of $n^0$, as the following result shows.

\begin{lem}
Let $x\in \R$ and assume that $f:\R \to \R$ has a finite number of roots, that we denote $x_1<x_2<...<x_N$. Let us denote
\begin{align*}
\O_0:=(-\infty, x_1),\qquad 
\O_i:=(x_i, x_{i+1}), \; \; i \in \{1,..., N-1\},  \qquad
\O_N:= (x_N, +\infty), 
\end{align*}  
and, among these segments, let us consider $\O_{i_1}, ... \O_{i_{M}}$ those which have an non-empty intersection with $\supp \left(n^0\right)$. Then, the set $\U:=\overline{ \bigcup\limits_{1\leq j\leq M}{\O_{i_j}}}$ and the family of sets $\left(\O_{i_j}\right)_{1\leq j\leq M}$ satisfy the hypotheses of Proposition \ref{prop Oi}. 
\label{Oi unidim}
\end{lem}

\begin{proof}
By applying the results stated at the end of Section \ref{section characteristic curves}, we note that for all $i\in \{1, ..., N\}$, $\O_i$ is positively invariant for the ODE $\dot x = f(x)$. Thus, for all $y\in \supp(n^0)\subset \U$, $t\geq 0$,  $X(t,y)\in \U$, which ensures that $X(\R_+\times \supp(n^0))\subset \U$. Moreover, the same results show that for all $j\in \{1,..., M\}$, $X(t, \supp(n^0)\cap \O_{i_j})\subset  \O_{i_j}$, and thus that $X(t, \supp(n^0))\cap \O_{i_j}\neq \emptyset$. The other two points are automatically satisfied, thanks to the definition of $\U$ and the sets $\O_i$. 
\end{proof}

Proposition \ref{prop Oi} shows us that $\rho$ satisfies ODE \eqref{main ODE}. Our next result shows that the long-time behaviour of this ODE depends on the long-time behaviour of the functions $R_i$. In particular, it states that if all the functions $R_i$ converge, then $\rho$ converges to the maximum of their limit. Before stating the result, we introduce some notations.

\begin{notation}
For any function $g:\R_+\to \R$, we denote:
 \[\underline{g}:=\underset{t\to +\infty}{\liminf} \; g(t) \quad { and } \quad  \overline{g}:=\underset{t\to +\infty}{\limsup} \; g(t),\]
and we say that $g$ converges to $l\in \R$ \textbf{with an exponential speed} if there exist $\delta>0$ such that 
\[ g(t)-l =\underset{t\to +\infty}{O}\left(e^{-\delta t}\right).\]
\end{notation}

\begin{prop}
The coupled system of ODEs~\eqref{main ODE} 
%\begin{align*}
%\begin{cases}
%\dot \rho_i(t)=\left(R_i(t)-\rho(t)\right)\rho_i(t)\quad  t\geq 0, \quad   i\in \{1,..., N\}\\
%\rho(t)=\sum\limits_{i=0}^N{\rho_i(t)}\quad  t\geq 0\\
%\rho_i(0)>0 \quad  i\in \{1,..., N \}\\
%\end{cases}
%\tag{\ref{main ODE}}
%\end{align*}
has the following properties: 
\begin{enumerate}[(i)]
\item For all $i\in \{1,..., N\}$ and all $t\geq 0$, $\rho_i(t)>0$. 
\item $\underline{\rho}\geq \underset{1\leq i\leq N}{\min}\; \left(\underline{R_i}\right)\quad $ and $\quad \overline{\rho}\leq \underset{1\leq i\leq N}{\max}\; \left(\overline{R_i}\right)$. 
\item Let $j\in \{1,..., N\} $. If there exists $i\in \{1,..., N\} $ such that $\overline R_j < \underline{R_i}$, then $\rho_j(t)\ul 0$. 
\item Let us assume that there exists $l\in \R_+\cup \{+\infty\}$, and a non empty set $I\subset \{1,..., N\}$ (where potentially $I=\{1,..., N\}$) such that for all $i\in I$, $R_i(t)\ul l$, and $\overline R_j<l$ for all $j\notin I$. Then, $\rho(t)\ul l$. 
\item Under the hypotheses of \textit{(iv)}, if moreover $0<l<+\infty$ and for all $i\in I$ the function $R_i$ converges to $l$ with an exponential speed, then $\rho$ converges to $l$ with an exponential speed. 
\end{enumerate}
\label{properties ODE}
\end{prop}
\begin{proof}
\begin{enumerate}[(i)]
\item According to the first line of ODE \eqref{main ODE}, $\rho_i(t)=e^{\int_0^t{R_i(s)-\rho(s)ds}}\rho_i(0)$, which is positive according to the third line. 
\item If $\underset{1\leq i\leq N}{\min}\; \left(\underline{R_i}\right) = 0$, there is nothing to prove: we assume $\underset{1\leq i\leq N}{\min}\; \left(\underline{R_i}\right) >0$ and let $m<\underset{1\leq i\leq N}{\min}\; \left(\underline{R_i}\right)$. There exists $T_m\geq 0$ such that for all $t\geq T_m$, and all $i\in \{1,..., N\}$, $R_i(t)\geq m$. Thus
\[\dot \rho(t)=\sum\limits_{i=1}^N{\dot \rho_i(t)}=\sum\limits_{i=1}^N{\left(R_i(t)-\rho(t)\right)\rho_i(t)}\geq \left(m-\rho(t)\right)\rho(t),\]
which means that $\rho$ is a super-solution of a logistic equation which converges to $m$, and thus that $\underline{\rho}\geq m$. Since this inequality holds for any $m<\underset{1\leq i\leq N}{\min}\; \left(\underline{R_i}\right)$ it proves that $\underline{\rho}\geq \underset{1\leq i\leq N}{\min}\; \left(\underline{R_i}\right)$. By proceeding in the same way with the limit superior, we get the second inequality. 
\item Let $i, j \in \{1,..., N\} $  such that $\overline R_j<\underline{R_i}$. The latter inequality is written with the convention that if $\underline{R_i} =+\infty$, then $\overline R_j \in \R$.  Using the first point, $\rho_j, \rho_i>0$ on $\R_+$. We can compute
\[\frac{d}{dt}\ln \left( \frac{\rho_i(t)}{\rho_j(t)}\right)=R_i(t)-R_j(t)>\varepsilon, \]
for a certain $\varepsilon>0$ and $t$ large enough. 
Thus, $\rho(t)\geq \rho_i(t)\geq Ce^{\e t }\rho_j(t)$, for a certain constant $C>0$, which yields
\[\dot \rho_j(t)\leq \left(\underset{t>0}{\sup}\; R_j(t)-Ce^{\e t}\rho_j(t)\right)\rho_j(t),\]
with $\underset{t>0}{\sup}\; R_j(t)<+\infty $ by hypothesis,  and thus $\rho_j$ goes to zero as $t$ goes to $+\infty$. 
\item Let us denote $\rho_{J}:=\sum\limits_{j\notin I}{\rho_j}$. (This first step is not necessary in the case $I=\{1,..., N\}$). According to the previous property, $\rho_{J}$ converges to zero. By denoting $\tilde{R_i}:=R_i-\rho_{J}$, we can thus rewrite system \eqref{main ODE} as:
\begin{align*}
\begin{cases}
\dot \rho_i(t)=\left(\tilde{R_i}(t)-\rho_I(t)\right)\rho_i(t)\quad \forall t\geq 0, \quad \forall i\in I\\
\rho_I(t)=\sum\limits_{i\in I}{\rho_i(t)}\quad \forall t\geq 0\\
\rho_i(0)>0 \quad \forall i \in \{1,..., N\}
\end{cases}. 
\end{align*}
Applying Property \textit{(ii)} to this new system proves the desired result, since 
\[\underset{i\in I}{\min}\; \left(\underline{\tilde{R_i}}\right)= \underset{i\in I}{\max}\; \left(\overline{\tilde{R}}_i\right)=l.\] 
\item Let $l\in (0, +\infty)$. According to the previous point, $\rho$ is bounded by two positive constants (and so is~$\rho_I)$, that we denote $\rho^m<\rho^M$. Using the same argument as in the proof of the third point, one proves that for all $j\notin I$, there exists $\e>0$ such that $\rho_J(t)\leq Ce^{-\varepsilon t} \rho^M$, and thus that $\rho_J$ converges to $0$ with an exponential speed. Thus, it remains to prove that the convergence of $\rho_I$ to $l$ also occurs with an exponential speed.  By hypothesis, there exists $C, \delta>0$ such that for all $t\geq 0$, $\sum\limits_{i\in I}{\lvert \tilde{R}_i(t)-l\rvert} \leq Ce^{-\delta t}$. Thus, by denoting $C':=C \lVert \rho_I(\cdot)-l \rVert_{\infty} \,\rho^M$, we find
%Let us denote $\tilde{\rho}_I(t):=\rho_I(t)-l$. This function satisfies
%\begin{align*}
%\dot{\tilde \rho}_I(t)=\sum\limits_{i\in I}{(\tilde{R}_i(t)-l-\tilde{\rho}_I(t))\rho_i(t)}, 
%\end{align*}
%and thus
%\[\rho^m(-Ce^{-\delta t}-\rho_I(t))\leq \dot{\tilde \rho}_I(t)\leq \rho^M(Ce^{-\delta t} -\rho_I(t)), \]
%since $\sum\limits_{i\in I}{\lvert \tilde{R}_i-l(t)\rvert} \leq Ce^{-\delta t}$,
%which concludes the proof, according to Grönwall's lemma. 
\begin{align*}
\frac{d}{dt}\frac{1}{2}\big( \rho_I(t)-l \big)^2 &= (\rho_I(t)-l) \sum\limits_{i\in I}{((\tilde{R}_i(t)-l)-({\rho}_I(t)-l))\rho_i(t)} \\
&\leq C' e^{-\delta  t }- \rho^m \left( \rho_I(t)-l \right)^2, 
\end{align*}
which concludes the proof, according to Grönwall's lemma. 
\end{enumerate}
\end{proof}

%{\color{red}
%\begin{rem}
%Let us assume that $S\in \mathcal{C}^\infty(\R_+)$, and that $S^{(k)}$ is positive for any  $k\in \N$, and let us denote $R_k:=\frac{S^{(k+1)}}{S^{(k)}}$ Then, since $R(t)=\frac{S'(t)}{S(t)}$, on can prove, by induction, that for every $k\in \N$, $R_k$ is a solution of the non-autonomous logistic equation 
%\[R_k'(t)=\left(R_{k+1}(t)-R_k(t)\right) R_k(t). \]
%Thus,  if there exists $k\in \N$ such that  $R_k$ converges to a certain $l\in \R$, then $R$ also converges to $l$. 
%\end{rem}
%}

\section{Results in the one-dimensional case}
\label{section unidimensional}
\subsection{Asymptotic behaviour of the carrying capacities}
\label{section Ri}
As evidenced by the previous section  and in particular by Proposition \ref{properties ODE}, the long-time behaviour of $\rho$ is completely determined by that of the functions $R_i$, which we call \textbf{carrying capacities} by analogy with the logistic equation. As their definition suggests, computing the limit of these functions is a delicate issue: this section is dedicated to these computations. The multidimensional case seems out of reach with this method, because, as we shall see, we use a change of variable that requires to be working in 1D. 

In order to simplify the notations, we will now denote $R$ instead of $R_\E$ or $R_i$, when there is no ambiguity as to which sets we are working with.  We are thus interested in the asymptotic behaviour of the function 
\begin{align}
R(t)=\frac{{\dot S(t)}}{S(t)}, \quad \text{with}\quad  S(t)=\int_{\E}{n^0(Y(t,x))e^{\int_0^{t}{\tr(Y(s,x))ds}}dx},  
\label{R(t)}
\end{align}
where  $\E\subset \R^d$ is an open set which satisfies $\supp(n^0)\cap Y(t, \E)\neq \emptyset$ for all $t\geq 0$. 

First, let us note that for all $l\in \R$, 
\begin{align}
R(t)-l=\frac{\frac{d}{dt}\left( S(t)e^{-lt}\right)}{S(t)e^{-lt}}. 
\label{R(t)-l}
\end{align}
Thus, in order to prove that $R$ converges to $l\in \R$ with an exponential speed, it in enough to prove that:
\begin{enumerate}[(a)]
\item $\underset{t\to +\infty}{\liminf}\,S(t)e^{-lt}>0.$
\item $ t\mapsto e^{\delta t}\frac{d}{dt}\left( S(t)e^{-lt}\right)$ is bounded for a certain $\delta >0$. 
\end{enumerate}
Indeed, we immediately deduce from \eqref{R(t)-l}, and the fact that $S$ is positive, according to its definition \eqref{Si}, that these two hypotheses imply that for any $\delta'\in (0, \delta)$,
\[ R(t)-l = \underset{t\to +\infty}{O}\left(e^{-\delta't}\right).  \]
\subsubsection{Integral formulae for the carrying capacities}
\label{expressions S}
This section aims at listing several alternative formulae of $S$. In the following section, we will use one or the other, depending on the studied case. 

We recall that $S$ is defined as 
\begin{align}
S(t)=\int_{\E}{n^0(Y(t,x))e^{\int_0^t{\tr(Y(s,x))ds}}dx}, 
\label{S init}
\end{align}
with $\tr:=r-\dive f$.

As seen in the first section, for any $t\geq 0$, $x\mapsto Y(t,x)$ is a $C^1$-diffeomorphism  from $\E$ to $Y(t,\E)$. Thus, the change of variable $y=Y(t,x)$, and Liouville's formula which ensures that $\lvert \det (\textrm{Jac}(Y(t,x))) \rvert= e^{\int_0^t{-\dive \; f(Y(s,x))ds}} $ provide a second expression for $S$, namely
\begin{align}
S(t)=\int_{Y(t,\E)}{n^0(y)e^{\int_0^t{r(X(s,y))ds}}dy}.
\label{Expression S y} 
\end{align}
Moreover, \textbf{in the one-dimensional case $x\in \R$}, if $\E$ is an interval on which $f\neq 0$, then for all $y\in \E$, $t\mapsto X(t,y)$ is also a $C^1$-diffeomorphism from $(0, t)$ to $(y, X(t,y))$ or $(X(t,y), y)$. This allows us to make the change of variable $s'=Y(s,x)$ and $s'=X(s,y)$ in the two expressions for $S$, thereby obtaining two new formulations
\begin{align}
S(t)=\int_{\E}{n^0(Y(t,x))e^{\int_{Y(t,x)}^{x}{\frac{\tr(s)}{f(s)}ds}}dx}=\int_{Y(t, \E)}{n^0(y)e^{\int_y^{X(t,y)}{\frac{r(s)}{f(s)}ds}}dy}, 
\label{S(t) Y(t,x)}
\end{align}
and, in the same way, for all $l\in \R$, 
\begin{align}
S(t)e^{-lt}=\int_{\E}{n^0(Y(t,x))e^{\int_{Y(t,x)}^{x}{\frac{\tr(s)-l}{f(s)}ds}}dx}=\int_{Y(t, \E)}{n^0(y)e^{\int_y^{X(t,y)}{\frac{r(s)-l}{f(s)}ds}}dy}.
\label{S uni}
\end{align}
Likewise, by differentiating expressions \eqref{S init} and \eqref{Expression S y}, we are led to several formulae for $\frac{d}{dt}\left(S(t)e^{-lt}\right)$, namely 
\begin{align}
\frac{d}{dt}\left(S(t)e^{-lt}\right)=\int_\E{m(Y(t,x))e^{\int_0^t{\tr(Y(s,x))-l\,ds}}dx}=\int_\E{m(y)e^{\int_0^t{r(X(s,y))-l\,ds}}dx},
\label{S prime}
\end{align}
with 
\begin{align}
m(y):=n^0(y)\left(\tr(y)-l\,\right)-f(y){n^0}'(y).
\label{m}
\end{align}
\textbf{In the one-dimensional case $x\in \R$}, assuming that $\E$ is an interval in which $f\neq 0$, we get the additional expressions 
\begin{align}
\frac{d}{dt}\left(S(t)e^{-lt}\right)=\int_\E{m(Y(t,x))e^{\int_{Y(t,x)}^x{\frac{\tr(s)-l}{f(s)}\,ds}}dx}=\int_\E{m(y)e^{\int_y^{X(t,y)}{\frac{r(s)-l}{f(s)}\,ds}}dy}.
\label{S prime unidim}
\end{align}
Lastly, in the particular one-dimensional case where $\E$ is an interval such that $Y(t,\E)=\E$  for all $t\geq 0$,  and $f\neq 0$ on $\E$, (which is the case if $\E$ is an interval delimited by two consecutive roots of $f$) one can differentiate~\eqref{Expression S y} to get
\begin{align}
\frac{d}{dt}\left(S(t)e^{-lt}\right)=\int_\E{n^0(y)\left(r(X(t,y))-l\right)e^{\int_0^t{r(X(s,y))-l\; ds}} dy}
\label{simple derivating y}
\end{align}
and the second expression of \eqref{S uni} to get 
\begin{align}
\frac{d}{dt}\left(S(t)e^{-lt}\right)=\int_{\E}{n^0(y)\left(r(X(t,y))-l\right)e^{\int_y^{X(t,y)}{\frac{r(s)-l}{f(s)}\,ds}}dy}=\int_\E{n^0(Y(t,x))(r(x)-l)e^{\int_{Y(t,x)}^x{\frac{\tr(s)-l}{f(s)}\,ds}}dx}.
\label{S prime 2}
\end{align}

\subsubsection{An important estimate}
The lemma stated in this section will be crucial in computing limits of the relevant parameter-dependent integrals in the next section. 

\begin{notation}
Let $x_0\in \R\cup\{\pm \infty\}$, and $h$ and $g$ be two functions defined in the neighbourhood of $x_0$. 
If there exist $C_1, C_2>0$ such that 
\[C_1\lvert g(x)\rvert \leq \lvert h(x)\rvert \leq C_2\lvert g(x)\rvert \]
for any $x$ close enough to $x_0$, we write
\[h(x)=\underset{x\to x_0}{\Theta}(g(x)).\]
\end{notation}

\begin{rem}
According to the definition of $\Theta$, is is clear that for any $x_0\in \R \cup \{\pm \infty\}$, $g,h$ defined in the neighbourhood of $x_0$, $f$ such that $h(x)=\underset{x\to x_0}{\Theta}(g(x))$, $h$ is integrable near $x_0$ if and only if $g$ is integrable near~$x_0$. 
\end{rem}

\begin{lem}
Let $x_0, y\in \R$, with $x_0\neq y$, and let  $\beta \in C^2([x_0, y])$ such that $\beta(y)=0$, $\beta'(y)\neq 0$ and $\beta \neq 0$ on $[x_0,y)$, and $\alpha\in C^1([x_0,y])$. 
Then, 
\[e^{\int_{x_0}^x{\frac{\alpha(s)}{\beta(s)}ds}}=\underset{x\to y}{\Theta}\left(\lvert y-x\rvert^{\frac{\alpha(y)}{\beta'(y)}}\right).\]
\label{alpha f}
\end{lem}
\begin{proof}
According to the regularity of $\alpha$ and $\beta$, for all $s\in (x_0, y)$, 
\[\alpha(s)=\alpha(y)+O(s-y),\quad \text{and } \quad\beta(s)= (s-y)\beta'(y)+O((s-y)^2) \]
Thus, 
\begin{align*}
\frac{\alpha(s)}{\beta(s)}-\frac{\alpha(y)}{(s-y)\beta'(y)}=\frac{\alpha(s)(s-y)\beta'(y)-\alpha(y)\beta(s)}{\beta(s)(s-y)\beta'(y)}=\frac{O(s-y)^2}{\beta'(y)^2(s-y)^2+O((s-y)^3)}=O(1).
\end{align*}
Hence, 
\begin{align*}
e^{\int_{x_0}^x{\frac{\alpha(s)}{\beta(s)}ds}}=e^{\int_{x_0}^x{\frac{\alpha(y)}{\beta'(y)}\frac{1}{s-y}+O(1)ds}}=e^{O(1)}\lvert y-x\rvert^{\frac{\alpha(y)}{\beta'(y)}}, 
\end{align*}
which proves the result of this lemma. 
\end{proof}

\subsubsection{Asymptotic behaviour of the carrying capacity in one dimension} 
\label{subsection R real}
We here focus on the one-dimensional case. 
We recall that we assume that $n^0\in \mathcal{C}_c(\R)$, $f\in C^2(\R)\cap \mathrm{Lip}(\R)$, $r\in C^1(\R)\cap L^1(\R)$, and that $r(x)$ goes to $0$ as $x$ goes to $\pm \infty$  In this section, we further assume that $f\in \mathrm{BV}(\R)$, \textit{i.e} $f' \in L^1(\R)$, and that $f$ converges to a non-zero limit at $\pm \infty$. 

In order to apply Proposition~\ref{properties ODE} (as explained in Lemma \ref{Oi unidim}), the most insightful division is to consider each segment between the roots of $f$. Hence, we must first compute the limit of the function $R$ when the chosen set $\E$ is such a segment. 
%delimited by two consecutive roots (or one root and $-\infty$ or $+\infty$. The following proposition does so in near full generality.

To be more precise, we must therefore distinguish between several cases, depending on whether the considered interval is bounded (delimited by two consecutive roots of $f$) or not (delimited by the smallest or the greatest root of $f$), and the sign of the derivative at these boundary roots. 

In fact, when $n^0$ vanishes at a given root $a$, the limit may depend on \textit{how fast} $n^0$ vanishes, \textit{i.e.} on the value $\alpha >0$ such that $n^0(y)$ vanishes like $(y-a)^\alpha$. For our method of proof to accommodate  this case, we will need to make a slightly stronger assumption involving the derivative of $n^0$.

We will see in the next section that a slight change in the limit of $R$ may have a drastic impact on the long-time behaviour of $n$.  We also deal with cases where $f$ does not have any root (which ensures, as one might expect, that $R$ converges to $0$), and the case where $f$ is zero on a whole interval. Hence, this result can be seen as a generalisation of the one stated in~\cite{lorenzi2020asymptotic}. 

\begin{prop}
In each case, we assume that $\E\cap \supp (n^0)\neq \emptyset$. \begin{enumerate}[(i)]
\item If $\E=(a,+\infty)$, $f<0$ on $\E$, $f(a)=0$ and $f'(a)<0$, then $R$ converges to $r(a)$.
\item If $\E=(-\infty, a)$, $f>0$ on $\E$, $f(a)=0$ and $f'(a)<0$, then  $R$ converges to $r(a)$.
\item If $\E=(a,+\infty)$, $f>0$ on $\E$,  $f(a)=0$, $f'(a)>0$, then
\begin{itemize}
\item If $n^0(a)>0$, then
\begin{itemize}
\item If $r(a)-f'(a)>0$,  then $R$ converges to $r(a)-f'(a)$. 
\item If $r(a)-f'(a)<0$,  then $R$ converges to $0$. 
\end{itemize}
\item If $n^0(a)=0$, and if there exist $C, \alpha>0$ such that ${n^0}'(y)=C \alpha (y-a)^{\alpha-1}+ \underset{y\to a^+}{O}((y-a)^\alpha)$, then
\begin{itemize}
\item If $r(a)-(1+\alpha)f'(a)>0$, then $R$ converges to $r(a)-(1+\alpha)f'(a)$.
\item If $r(a)-(1+\alpha)f'(a)<0$, then $R$ converges to $0$. 
\end{itemize}
\item If $n^0(a)=0$, and if there exists $\e>0$ such that $n^0(\cdot)=0$ on $[a, a+\e]$, then $R$ converges to $0$. 
\end{itemize}
\item If $\E=(-\infty,a)$, $f<0$ on $\E$,  $f(a)=0$, $f'(a)>0$, then
\begin{itemize}
\item If $n^0(a)>0$, then
\begin{itemize}
\item If $r(a)-f'(a)>0$,  then $R$ converges to $r(a)-f'(a)$. 
\item If $r(a)-f'(a)<0$,  then $R$ converges to $0$. 
\end{itemize}
\item If $n^0(a)=0$, and if there exist $C, \alpha>0$ such that ${n^0}'(y)=-C \alpha (a-y)^{\alpha-1}+ \underset{y\to a^-}{O}((a-y)^\alpha)$, then
\begin{itemize}
\item If $r(a)-(1+\alpha)f'(a)>0$, then $R$ converges to $r(a)-(1+\alpha)f'(a)$.
\item If $r(a)-(1+\alpha)f'(a)<0$, then $R$ converges to $0$. 
\end{itemize}
\item If $n^0(a)=0$, and if there exists $\e>0$ such that $n^0(\cdot)=0$ on $[a-\e, a]$, then $R$ converges to $0$.
\end{itemize}
\item If $\E=(a,b)$, $f>0$ on $(a,b)$, $f(a)=f(b)=0$,  $f'(a)>0$, $f'(b)<0$, then 
\begin{itemize}
\item If $n^0(a)>0$, then
\begin{itemize}
\item If $r(b)>r(a)-f'(a)$, then $R$ converges to $r(b)$.
\item If $r(b)<r(a)-f'(a)$, then $R$ converges to $r(a)-f'(a)$.
\end{itemize}
\item If $n^0(a)=0$, and if there exist $C, \alpha>0$ such that ${n^0}'(y)=C \alpha (y-a)^{\alpha-1}+ \underset{y\to a^+}{O}((y-a)^\alpha)$, then
\begin{itemize}
\item If $r(b)>r(a)-(\alpha+1)f'(a)$, then $R$ converges to $r(b)$. 
\item If $r(b)<r(a)-(\alpha+1)f'(a)$, then $R$ converges to $r(a)-(\alpha+1)f'(a)$. 
\end{itemize}
\item If $n^0(a)=0$, and if there exists $\e>0$ such that $n^0(\cdot)=0$ on $[a, a+\e]$, then $R$ converges to $r(b)$.
\end{itemize}
\item If $\E=(a,b)$, $f<0$ on $(a,b)$, $f(a)=f(b)=0$,  $f'(a)<0$, $f'(b)>0$, then 
\begin{itemize}
\item If $n^0(b)>0$, then
\begin{itemize}
\item If $r(a)>r(b)-f'(b)$, then $R$ converges to $r(a)$.
\item If $r(a)<r(b)-f'(b)$, then $R$ converges to $r(b)-f'(b)$.
\end{itemize}
\item If $n^0(b)=0$, and if there exist $C, \alpha>0$ such that ${n^0}'(y)=-C \alpha (b-y)^{\alpha-1}+ \underset{y\to b^-}{O}((b-y)^\alpha)$, then
\begin{itemize}
\item If $r(a)>r(b)-(\alpha+1)f'(b)$, then $R$ converges to $r(a)$. 
\item If $r(a)<r(b)-(\alpha+1)f'(b)$, then $R$ converges to $r(b)-(\alpha+1)f'(b)$. 
\end{itemize}
\item If $n^0(b)=0$, and if there exists $\e>0$ such that $n^0(\cdot)=0$ on $[b-\e, b]$, then $R$ converges to $r(a)$. 
\end{itemize}
\item If $\E=\R$, and $f>0$ on $\R$, then $R$ converges to $0$. 
\item If $\E=\R$, and $f<0$ on $\R$, then $R$ converges to $0$. 
\item If $\E$ is a interval in which $f\equiv 0$, and $n^0>0$, and $\underset{\bar \E}{\arg \max}\; r=\{x_1,..., x_p\}\subset \E$, with 
\[r'(x_i)=0, \quad r''(x_i)<0\quad \text{for all } i\in \{1,..., p\}, \]
then, $R$ converges to $\bar r :=\underset{x\in \E}{\max}\; r(x)$. 
\end{enumerate}
Moreover, except in this last case, $R$ converges with an exponential speed whenever it does not converge to~$0$. 
\label{limit Ri}
\end{prop}

\begin{proof}
As explained at the beginning of this section, whenever we show that $R$ converges with an exponential speed, we must prove successively that
\begin{enumerate}[(a)]
\item $\underset{t\to +\infty}{\liminf}\,S(t)e^{-lt}>0$
\item $ t\mapsto e^{\delta t}\frac{d}{dt}\left( S(t)e^{-lt}\right)$ is bounded for a certain $\delta >0$, 
\end{enumerate} 
where $l$ is the expected limit. By Fatou's lemma, the point (a) can be proven by showing that the integrand involved in the expression of $S$ (which depends on the chosen formula) converges pointwise to a non-negative function which is positive on a set of positive Lebesgue measure. 
Depending on the case, we will use different expressions for $S$ and $S'$ among those determined in Section \ref{expressions S}. In order to lighten the proof, we assume without loss of generality that $a=0$ and $b=1$, and we denote 
$$\tr=r-f' \quad \text{and}\quad \tr_\alpha:= \tr-\alpha f'= r-(\alpha+1)f' \quad \text{for } \alpha \in \R. $$ Moreover since the cases \textit{(ii)}, \textit{(iv)}, \textit{(vi)} and \textit{(viii)} are symmetric to the cases \textit{(i)}, \textit{(iii)}, \textit{(v)} and \textit{(vii)} respectively, we omit their proof. 
\begin{enumerate}[(i)]
\item[(i)] Note that, according to the hypotheses satisfied by $f$, for all $y\in(0, +\infty)$,  $t\mapsto X(t,y)$ converges to $0$. 
\begin{enumerate} 
\item According to \eqref{Expression S y}, 
\[S(t)e^{-r(0)t}= \int_0^{+\infty}{n^0(y)e^{\int_0^t{r(X(s,y))-r(0)ds}}dy}= \int_0^{M}{n^0(y)e^{\int_0^t{r(X(s,y))-r(0)ds}}dy},\] 
for a certain $M>0$, since $n^0$ has a compact support. Since $f'(0)<0$, there exist $C, \delta>0$ such that $X(t,y)\leq Ce^{-\delta t}$ for all $y\in [0,M]$, $t\geq 0$. This proves that for all $y\in [0,M]$, $s\mapsto r(X(s,y))-r(0)$ is integrable on $(0, +\infty)$, and thus that $y\mapsto n^0(y)e^{\int_0^{+\infty}{r(X(s,y))-r(0)ds}}$ is well-defined on $[0, M]$. Since this function is positive on a sub-interval of $[0,M]$, its integral on this segment is positive.  Moreover,  $t\mapsto n^0(y)e^{\int_0^t{r(X(s,y))-r(0)ds}}$ converges pointwise to this function.
%therefore, Fatou's lemma ensures that that $\underset{t\to +\infty}{\liminf}\,S(t)e^{-r(0)t}>0$. 
%%
\item As seen in the first point, there exist $C, \delta>0$ such that for all $y\in [0,M]$ and all $t\geq 0$, $0\leq X(t,y)\leq Ce^{-\delta t}$. 
Thus, using expression \eqref{simple derivating y}, and the mean value theorem,  
\begin{align*}
\bigg\lvert e^{\delta t}\frac{d}{dt}\left(S(t)e^{-r(0)t}\right)\bigg\rvert &=e^{\delta t }\bigg\lvert \int_0^M{n^0(y) \big(r(X(t,y))-r(0)\big) e^{\int_0^t{r(X(t,y))-r(0)ds}}dy} \bigg\rvert \\
&\leq 2 \lVert n^0 \rVert_{\infty}  \lVert r \rVert_{L^\infty(0, M)}C\int_0^M{e^{\int_0^t{\lvert r(X(s,y))-r(0) \rvert ds }}dy}\\
&\leq  \lVert n^0 \rVert_{\infty} \lVert r' \rVert_{L^\infty(0, M)}CM e^{\int_0^t{C\lVert r' \rVert_{L^\infty(0, M)}e^{-\delta s} ds}}
\end{align*}
which is bounded. 
\end{enumerate}
\item[(iii)] Note that, according to the hypothesis on $f$, for all $x,y \in (0, +\infty)$, $t\mapsto X(t,y)$ is increasing and goes to $+\infty$, and $t\mapsto Y(t,x)$ is decreasing and converges to $0$.  
\begin{itemize}
\item Let us assume that $n^0(0)>0$. We distinguish two cases:
\begin{itemize}
\item \textbf{Case $r(0)-f'(0)>0$:} 
\begin{enumerate}[(a)]
\item According to \eqref{S uni}, 
$$S(t)e^{-\tr(0)t}=\int_{\E}{n^0(Y(t,x))e^{\int_{Y(t,x)}^{x}{\frac{\tr(s)-\tr(0)}{f(s)}ds}}dx}. $$
For all $x\in (0, +\infty)$, $n^0(Y(t,x))e^{\int_{Y(t,x)}^x{\frac{\tr(s)-\tr(0)}{f(s)}ds}}\underset{t\to+\infty}{\longrightarrow}n^0(0)e^{\int_{0}^x{\frac{\tr(s)-\tr(0)}{f(s)}ds}}$, which is well defined since $s\mapsto \frac{\tr(s)-\tr(0)}{f(s)}$ is continuous on $[0, x)$, thanks to the regularity of $r$ and $f$, and positive, since $n^0(0)>0$ by hypothesis. %We conclude with Fatou's lemma. 
\item Let $\delta \in \big( 0, \min(\tr(0), f'(0))\big)$. 
Since $\delta-\tr(0)<0$, $r$ goes to $0$ at $+\infty$ and $f$ is positive, we can find $M\geq 0$  such that $\frac{r(s)-\tr(0)+\delta}{f(s)}\leq 0$ for all $s\in [M, +\infty)$, and $\supp \left(n^0\right)\cap \E\subset [0, M]$. Thus, for all $t\geq 0$, and all  $y\in (0, M)$,  
\[\int_y^{X(t,y)}{\frac{r(s)-\tr(0)+\delta}{f(s)}ds}\leq \int_y^M{\bigg\lvert \frac{r(s)-\tr(0)+\delta}{f(s)} \bigg\rvert ds  }. \]
According to \eqref{S prime unidim}, 
\[e^{\delta t }\frac{d}{dt}\left(S(t) e^{-\tr(0) t }\right)=\int_{0}^{+\infty}{m(y)e^{\int_y^{X(t,y)}{\frac{r(s)-\tr(0)+\delta }{f(s)}ds}}dy}. \]
Thus, since $\supp(m)\cap \E=\supp(n^0)\cap \E\subset [0, M]$, and by  the previous inequality, 
\[\bigg\lvert e^{\delta t }\frac{d}{dt}\left(S(t) e^{-\tr(0) t }\right) \bigg\rvert\leq \int_0^M{\lvert m(y)\rvert e^{\int_y^M{ \frac{\lvert r(s)-\tr(0)+\delta \rvert }{f(s)}  ds }}dy}.  \]
Since $m(y)=n^0(y)(\tr(y)-\tr(0))-f(y){n^0}'(y)$, $\lvert m(y)\rvert= \underset{y\to 0}{O}(y)$. Moreover, since $\lvert r(0)-\tr(0)+\delta \rvert= f'(0)+\delta$, Lemme~\ref{alpha f} yields $e^{\int_y^M{\frac{\lvert r(s)-\tr(0)+\delta \rvert }{f(s)}ds}}=\underset{y\to 0}{O}(y^{-1-\delta/f'(0)})$. Therefore, 
\[\lvert m(y) \rvert  e^{\int_y^M{\frac{\lvert r(s)-\tr(0)+\delta \rvert }{f(s)}ds}}=  \underset{y\to 0}{O}(y^{-\delta/f'(0)}),   \]
and is thus integrable since $\delta<f'(0)$. 
\end{enumerate}
\item \textbf{Case $r(0)-f'(0)<0$:} in this case, we do not show that convergence occurs with an exponential speed. Thus, we do not prove the two points as before, but simply that $$\underset{t \to +\infty}{\limsup}\; S(t)>0 \quad \text{and} \quad  \underset{t\to +\infty }{\lim }S'(t)=0, $$
which will imply,  by definition of $R$ \eqref{R(t)}, that $R$ converges to $0$. 
 According to \eqref{S uni}, 
\[S(t)= \int_0^{+\infty}{n^0(y)e^{\int_y^{X(t,y)}{\frac{r(s)}{f(s)}ds}}dy}. \]
By hypothesis, $f$ converges to a positive limit. Thus, for all $y>0$, there exist $\e_y>0$ such that $f(s)> \e_y$, for all $s\geq y$. 
Thus, for all $y>0$, $\int_y^{+\infty}{\frac{r(s)}{f(s)}ds}\leq \frac{1}{\e_y} \lVert r \rVert_{L^1}<+\infty $. 
This implies that $y\mapsto n^0(y)e^{\int_y^{+\infty}{\frac{r(s)}{f(s)}ds}}$ is well defined on $\R_+$. Moreover, this function is positive at any~$y$ such that $n^0(y)>0$, hence its integral is positive. Finally, $t \mapsto n^0(y)e^{ \int_y^{X(t,y)}{\frac{r(s)}{f(s)}ds}}$ converges to this function pointwise,.%Fatou's lemma yields $\underset{t\to +\infty}{\liminf}\,S(t)>0$. 

Owing to \eqref{simple derivating y} (with $l=0$),  
\[S'(t)=\int_{\supp(n^0)}{n^0(y) r(X(t,y))e^{\int_{y}^{X(t,y)}{\frac{r(s)}{f(s)}ds}}dy}. \]
By hypothesis, there exist $\e, M>0$ such that $f(s)\geq \e$ for all $s\geq M$. Thus, for all $y>0$, 
\[\int_y^{X(t,y)}{\frac{r(s)}{f(s)}ds}\leq \int_y^{+\infty}{\frac{r(s)}{f(s)}ds}\leq \underbrace{\int_M^{+\infty}{\frac{r(s)}{f(s)}ds}}_{\leq \frac{\lVert r\rVert_{L^1}}{ \e}}+ \int_y^M{\frac{r(s)}{f(s)}ds}\; \ind_{(0, M)}(y). \]
Since $r\in L^1(\R_+)$, by hypothesis, this proves that there exists a constant $K>0$ such that for all $t\geq 0$, $y>0$, 
\[\bigg\lvert n^0(y)r(X(t,y))e^{\int_y^{X(t,y)}{\frac{r(s)}{f(s)}ds}}\bigg\rvert \leq \lVert n^0 \rVert_{\infty} \lVert r \rVert_{\infty}e^K e^{\int_y^M{\frac{r(s)}{f(s)}ds}\; \ind_{(0, M)}(y)}.\]
By virtue of Lemma \ref{alpha f}, this last quantity in integrable, since 
\[e^{\int_y^M{\frac{r(s)}{f(s)}ds}}=\underset{y\to 0}{O} \left(y^{-\frac{r(0)}{f'(0)}}\right), \]
with $r(0)<f'(0)$, by hypothesis. Moreover, since $t\mapsto r(X(t,y))$ converges to $0$ as $t$ goes to $+\infty$ for any $y>0$, $n^0(y) \, r(X(t,y)) \,e^{\int_{y}^{X(t,y)}{\frac{r(s)}{f(s)}ds}}dy$ converges to $0$ pointwise. According to the dominated convergence theorem, $S'$ thus converges to $0$. 
\end{itemize}
\item Let us assume that $n^0(a)=0$, and that the hypothesis of the theorem regarding ${n^0}'$ holds. We follow exactly the same steps and use the same formulae as in the case `$n^0(0)>0$', by adapting the computations. We distinguish again two cases.
\begin{itemize}
\item \textbf{Case $\tr_\alpha(0)>0$:} 
\begin{enumerate}[(a)]
\item According to \eqref{S uni}, 
$$S(t)e^{-\tr_\alpha(0)t}=\int_{\E}{n^0(Y(t,x))e^{\int_{Y(t,x)}^{x}{\frac{\tr(s)-\tr(0)}{f(s)}ds}}dx}. $$
For all $x\in (0, +\infty)$, 
\begin{align*}
n^0(Y(t,x))e^{\int_{Y(t,x)}^x{\frac{\tr(s)-\tr_\alpha(0)}{f(s)}ds}}&= \frac{n^0(Y(t,x))}{Y(t,x)^\alpha}e^{\int_{Y(t,x)}^x{\frac{\tr(s)-\tr(0)}{f(s)}ds}}\; Y(t,x)^\alpha e^{\int_{Y(t,x)}^x{\frac{\alpha f'(0)}{f(s)}}ds}.
\end{align*}
Let $x>0$. On the one hand,
\[\frac{n^0(Y(t,x))}{Y(t,x)^\alpha}e^{\int_{Y(t,x)}^x{\frac{\tr(s)-\tr(0)}{f(s)}ds}} \underset{t\to+\infty}{\longrightarrow}C e^{\int_{0}^x{\frac{\tr(s)-\tr(0)}{f(s)}ds}} \]
which is well defined since $s\mapsto \frac{\tr(s)-\tr(0)}{f(s)}$ is continuous on $[0, x)$, according to the regularity assumptions on $r$ and $f$, and positive. On the other hand, by rewriting
\begin{align*}
Y(t,x)^\alpha e^{\int_{Y(t,x)}^x{\frac{-\alpha f'(0)}{f(s)}}ds}&= e^{\alpha (\ln(Y(t,x))-\ln(x))}x^\alpha e^{\int_{Y(t,x)}^x{\frac{\alpha f'(0)}{f(s)}}ds}\\
&=x^\alpha e^{\int_{Y(t,x)}^x{\frac{\alpha f'(0)}{f(s)}-\frac{\alpha}{s}ds}}, 
\end{align*}
and by noting that $s\mapsto \frac{\alpha f'(0)}{f(s)}-\frac{\alpha}{s}$ is continuous at $0$, since 
\[\frac{\alpha f'(0)}{f(s)}-\frac{\alpha}{s}= \frac{\alpha f'(0)s-\alpha f(s)}{sf(s)}=\frac{\alpha f'(0)s- \alpha f'(0)s+f''(0)/2 s^2+o(s^2) }{f'(0)s^2+o(s^2)}\underset{s\to 0}\longrightarrow -\frac{\alpha f''(0)}{2f'(0)}, \]
we show that 
\[Y(t,x)^\alpha e^{\int_{Y(t,x)}^x{\frac{-\alpha f'(0)}{f(s)}}ds}\underset{t\to +\infty}{\longrightarrow} x^\alpha e^{\int_0^x{\frac{\alpha f'(0)}{f(s)}-\frac{\alpha}{s}ds}}, \]
which is also well defined, and positive.  
%We conclude this point by applying Fatou's lemma. 
\item Let $\delta \in \big( 0, \min(\tr_\alpha (0), f'(0))\big)$. 
Since $\delta-\tr_\alpha(0)<0$, $r$ goes to $0$ at $+\infty$ and $f$ is positive, we can find $M\geq 0$  such that $\frac{r(s)-\tr_\alpha(0)+\delta}{f(s)}\leq 0$ for all $s\in [M, +\infty)$, and $\supp \left(n^0\right)\subset [0, M]$. Thus, for all $t\geq 0$, and all  $y\in (0, M)$,  
\[\int_y^{X(t,y)}{\frac{r(s)-\tr_\alpha(0)+\delta}{f(s)}ds}\leq \int_y^M{\bigg\lvert \frac{r(s)-\tr_\alpha(0)+\delta}{f(s)} \bigg\rvert ds  }. \]
According to \eqref{S prime unidim}, 
\[e^{\delta t }\frac{d}{dt}\left(S(t) e^{-\tr_\alpha(0) t }\right)=\int_{0}^{+\infty}{m(y)e^{\int_y^{X(t,y)}{\frac{r(s)-\tr_\alpha(0)+\delta }{f(s)}ds}}dy}. \]
Thus, since $\supp(m)\cap \E=\supp(n^0)\cap \E\subset [0, M]$, and thanks to  the previous inequality, 
\[\bigg\lvert e^{\delta t }\frac{d}{dt}\left(S(t) e^{-\tr_\alpha(0) t }\right) \bigg\rvert\leq \int_0^M{\lvert m(y)\rvert e^{\int_y^M{ \frac{\lvert r(s)-\tr_\alpha (0)+\delta \rvert }{f(s)}  ds }}dy}.  \]
Let us prove that this integral is bounded. First, let us note that  $$m(y)=n^0(y)(\tr(y)-\tr_\alpha(0))-f(y){n^0}'(y)=\underset{y\to 0^+}O(y^{\alpha+1}) $$ 
Indeed, since $n^0(y)=Cy^{\alpha}+ \underset{y\to 0^+}{O}(y^{\alpha+1})$ and ${n^0}'(y)= C\alpha y^{\alpha-1}+ \underset{y\to 0^+}{O}(y^\alpha)$, 
\begin{align*}
\frac{\lvert m(y)\rvert }{y^{\alpha+1}}&\leq \frac{n^0(y)}{y^\alpha}\frac{\lvert \tr(y)-\tr(0)\rvert }{y}+ \frac{\lvert \alpha f'(0)n^0(y)-f(y){n^0}'(y) \rvert}{y^{\alpha+1}}\\
&\leq \frac{n^0(y)}{y^\alpha}\lVert {\tr}' \rVert_{\infty}+ \frac{\lvert C \alpha f'(0)y^{\alpha} -C \alpha f'(0) y^{\alpha} + O(y^{\alpha+1})\rvert }{y^{\alpha+1}}\\
&=\underset{y\to 0^+}{O}(1). 
\end{align*}
Moreover, according to Lemma \ref{alpha f}, since $ \lvert r(0)-\tr_\alpha(0)+\delta \rvert= (\alpha+1 )f'(0)+\delta$, $$e^{\int_y^M{\frac{\lvert r(s)-\tr_\alpha(0)+\delta \rvert }{f(s)}ds}}=\underset{y\to 0^+}{O}(y^{-\alpha-1-\delta/f'(0)}).$$
 Therefore, 
\[\lvert m(y) \rvert  e^{\int_y^M{\frac{\lvert r(s)-\tr(0)+\delta \rvert }{f(s)}ds}}=  \underset{y\to 0^+}{O}(y^{-\delta/f'(0)}),   \]
and is thus integrable since $\delta<f'(0)$. 
\end{enumerate}
\item \textbf{Case $\tr_\alpha(0)<0$:} again, we just prove that  $$\underset{t \to +\infty}{\limsup}\; S(t)>0 \quad \text{and} \quad  \underset{t\to +\infty }{\lim }S'(t)=0. $$
According to \eqref{S uni}, 
\[S(t)= \int_0^{+\infty}{n^0(y)e^{\int_y^{X(t,y)}{\frac{r(s)}{f(s)}ds}}dy}. \]
By hypothesis, $f$ converges to a positive limit. Thus, for all $y>0$, there exist $\e_y>0$ such that $f(s)> \e_y$, for all $s\geq y$. 
Hence, for all $y>0$, $\int_y^{+\infty}{\frac{r(s)}{f(s)}ds}\leq \frac{1}{\e_y} \lVert r \rVert_{L^1}<+\infty $. 
This ensures that $y\mapsto n^0(y)e^{\int_y^{+\infty}{\frac{r(s)}{f(s)}ds}}$ is well defined on $\R_+$. Moreover, this function is positive for every $y$ such that $n^0(y)>0$, which ensures that its integral is positive, and $t \mapsto n^0(y)e^{ \int_y^{X(t,y)}{\frac{r(s)}{f(s)}ds}}$ converges to this function pointwise.
%Fatou's lemma yields $\underset{t\to +\infty}{\liminf}S(t)>0$. 
%
According to \eqref{simple derivating y}, (with $l=0$),  
\[S'(t)=\int_{\supp(n^0)}{n^0(y) r(X(t,y))e^{\int_{y}^{X(t,y)}{\frac{r(s)}{f(s)}ds}}dy}. \]
By hypothesis, there exist $\e, M>0$ such that $f(s)\geq \e$ for all $s\geq M$. Thus, for all $y>0$, 
\[\int_y^{X(t,y)}{\frac{r(s)}{f(s)}ds}\leq \int_y^{+\infty}{\frac{r(s)}{f(s)}ds}\leq \underbrace{\int_M^{+\infty}{\frac{r(s)}{f(s)}ds}}_{\leq \frac{\lVert r\rVert_{L^1}}{ \e}}+ \int_y^M{\frac{r(s)}{f(s)}ds}\; \ind_{(0, M)}(y). \]
Since $r\in L^1(\R_+)$, this proves that there exist a constant $K>0$ such that for all $t\geq 0$, $y>0$, 
\[\bigg\lvert n^0(y)r(X(t,y))e^{\int_y^{X(t,y)}{\frac{r(s)}{f(s)}ds}}\bigg\rvert \leq  \lVert r \rVert_{\infty}e^K n^0(y)e^{\int_y^M{\frac{r(s)}{f(s)}ds}\; \ind_{(0, M)}(y)}.\]
By hypothesis, and according to Lemma \ref{alpha f},
\[n^0(y)=\underset{y\to 0^+}{O}(y^{\alpha}) \quad \text{and} \quad e^{\int_y^M{\frac{r(s)}{f(s)}ds}}=\underset{y\to 0^+}{O} \left(y^{-\frac{r(0)}{f'(0)}}\right). \]
Thus, 
\[n^0(y)e^{\int_y^M{\frac{r(s)}{f(s)}ds}}=\underset{y\to 0^+}{O}\left(y^{\alpha-\frac{r(0)}{f'(0)}}\right), \]
with $\alpha-\frac{r(0)}{f'(0)}>-1$.  
Moreover, since $t\mapsto r(X(t,y))$ converges to $0$ as $t$ goes to $+\infty$ for any $y>0$, $n^0(y)\, r(X(t,y))\, e^{\int_{y}^{X(t,y)}{\frac{r(s)}{f(s)}ds}}dy$ converges to $0$ pointwise. By the dominated convergence theorem, $S'$ thus converges to $0$. 
\end{itemize}
\item We can prove this point exactly as we treat the case $f>0$ on $\R$. We therefore leave it to the reader and refer to the proof of \textit{(vii)}. 
\end{itemize}
\item[(v)]  Let us note that, for any $x,y \in (0,1)$, $t\mapsto X(t,y)$ is increasing and converges to $1$, and $t\mapsto Y(t,x)$ is decreasing and converges to $0$. 
\begin{itemize}
\item Let us assume that $n^0(a)>0$. We distinguish again between two cases:
\begin{itemize}
\item \textbf{Case $r(1)>\tr(0)$:}
\begin{enumerate}[(a)]
\item Let us use the second expression \eqref{S uni} for $S$, \textit{i.e.}
\[S(t)e^{-r(1)t}=\int_{0}^{1}{e^{\int_y^{X(t,y)}{\frac{r(s)-r(1)}{f(s)}ds}}dy}. \]
 For all $y\in (0,1)$, $n^0(y)e^{\int_y^{X(t,y)}{\frac{r(s)-r(1)}{f(s)}ds}}\underset{t\to +\infty}{\longrightarrow}n^0(y)e^{\int_y^1{\frac{r(s)-r(1)}{f(s)}ds}}$, which is well-defined for all $y\in (0,1)$, since $s\mapsto \frac{r(s)-r(1)}{f(s)}$ is continuous on $(0, 1]$, and positive on a set of non-zero measure, since it is positive where $n^0$ is positive. 
 %Thus, according to Fatou's lemma, 
%\begin{align*}
%\underset{t\to +\infty}{\liminf}\;S(t)e^{-r(1)t}>0. 
%\end{align*}
%
\item Let $\delta\in \big(0,\min\left(r(1)-\tr(0), -f'(1)\right)\big)$ . 
Since $\tr(0)-r(1)+\delta<0$, there exists $m\in(0,1)$ such that $\tr(s)-r(1)+\delta$ for all $s \in (0,m]$. Thus, for all $x\in (0,1)$, $t\geq 0$,   
\begin{align*}
\int_{Y(t,x)}^{x}{\frac{\tr(s)-r(1)+\delta}{f(s)}ds}&\leq \int_{m}^x{\frac{\lvert \tr(s)-r(1)+\delta\rvert }{f(s)}ds}\; \ind_{(m,1)}(x). 
\end{align*}
Thus, using expression \eqref{S prime 2},
\begin{align*}
e^{\delta t} \frac{d}{dt}\left(S(t)e^{-r(1)t}\right)&=\int_0^1{n^0\left(Y(t,x)\right)\left(r(x)-r(1)\right)e^{\int_{Y(t,x)}^x{\frac{\tr(s)-r(1)+\delta}{f(s)}ds}}dx}\\
&\leq \lVert n^0 \rVert   \int_0^1{\lvert r(x)-r(1)  \rvert e^{\int_m^x{\frac{\lvert \tr(s)-r(1)+\delta\rvert}{f(s)}ds}\; \ind_{(m,1)}(x)}dx}<+\infty.
\end{align*}
This last integral is finite since $\lvert \tr(1)-r(1)+\delta\rvert = -f'(1)+\delta$, and thus $e^{\int_a^x{\frac{\lvert \tr(s)-\tr(0)\rvert }{f(s)}ds}}=\underset{x\to 1}O\left(\lvert x-1\rvert ^{\frac{\delta}{f'(1)}-1}\right)$, by Lemma \ref{alpha f}) $\lvert r(x)-r(1)\rvert =\underset{x\to 1}{O}\lvert x-1 \rvert $, and $\frac{\delta}{f'(1)}>-1$ by hypothesis. 
\end{enumerate}
\item \textbf{Case $\tr(0)>r(1)$:}
\begin{enumerate}[(a)]
\item Using \eqref{S uni}, we find 
\[S(t)e^{-\tr(0)t}=\int_0^1{n^0(Y(t,x))e^{\int_{Y(t,x)}^{x}{\frac{\tr(s)-\tr(0)}{f(s)}ds}}dx}. \]
For all $x\in (0,1)$, $n^0(Y(t,x))e^{\int_{Y(t,x)}^x{\frac{\tr(s)-\tr(0)}{f(s)}}ds}\underset{t\to +\infty}{\longrightarrow}{n^0(0)e^{\int_{0}^x{\frac{\tr(s)-\tr(0)}{f(s)}}ds}}$, 
which is well-defined since $s\mapsto \frac{\tr(s)-\tr(0)}{f(s)}$ is continuous on $[0, 1)$, and positive by hypothesis on $n^0$. %Thus, according to Fatou's lemma, $\underset{t\to +\infty}{\liminf}\;S(t)e^{-\tr(0)t}>0$. 
\item Let $\delta\in \big(0, \min(\tr(0)-r(1), f'(0))\big)$.
Since $r(1)-\tr(0)+\delta<0$, there exists $M\in (0,1)$ such that $r(s)-\tr(0)+\delta<0$ for all $s\geq M$. Thus, for all $y\in (0,1)$, $t\geq 0$,  $\int_y^{X(t,y)}{\frac{r(s)-\tr(0)+\delta}{f(s)}}\leq \int_y^M{\frac{\lvert r(s)-\tr(0)+\delta \rvert }{f(s)}ds}\; \ind_{(0,M)}(y)$. 
Thus, 
\begin{align*}
\bigg\lvert e^{\delta t } \frac{d}{dt}\left( S(t)e^{-\tr(0) t }\right)\bigg\rvert =\bigg\vert \int_0^1{m(y)e^{\int_{y}^{X(t,y)}{\frac{r(s)-\tr(0)+\delta}{f(s)}ds}}}\bigg\rvert &\leq  \int_0^1{\lvert m(y)\rvert e^{\int_{y}^{M}{ \frac{\lvert r(s)-\tr(0)+\delta\rvert}{f(s)}ds}\; \ind_{(0,M)}(y)}dy},
\end{align*}
which is a finite integral, since $\lvert r(0)-\tr(0)+\delta \rvert=f'(0)+\delta$, and thus $e^{\int_{y}^{b}{ \frac{\lvert r(s)-\tr(0)+\delta\rvert}{f(s)}ds}}=\underset{y\to 0}{O}\left(y^{-\frac{\delta}{f'(0)}-1}\right)$ (by Lemma \ref{alpha f}),  $m(y)=n^0(y)\left(\tr(y)-\tr(0)\right)-f(y){n^0}'(y)=\underset{y\to 0}{O}\left(y \right)$, and $\frac{\delta}{f'(0)}<1$ thanks to our choice for $\delta$. 
\end{enumerate}
\end{itemize}
\item   Let us assume that $n^0(a)=0$, and that the hypothesis on ${n^0}'$ of the theorem holds. As usual, we distinguish two cases. 
\begin{itemize}
\item \textbf{Case $r(1)>\tr_\alpha(0)$:} 
\begin{enumerate}[(a)]
\item  This first point is exactly the same as in the case $n^0>0$.  Let us use the second expression~\eqref{S uni} for $S$, \textit{i.e.}
\[S(t)e^{-r(1)t}=\int_{0}^{1}{e^{\int_y^{X(t,y)}{\frac{r(s)-r(1)}{f(s)}ds}}dy}. \]
 For all $y\in (0,1)$, $n^0(y)e^{\int_y^{X(t,y)}{\frac{r(s)-r(1)}{f(s)}ds}}\underset{t\to +\infty}{\longrightarrow}n^0(y)e^{\int_y^1{\frac{r(s)-r(1)}{f(s)}ds}}$, which is well-defined for all $y\in (0,1)$, since $s\mapsto \frac{r(s)-r(1)}{f(s)}$ is continuous on $(0, 1]$, and positive on a set of measure non-zero, since it is positive where $n^0$ is positive. 
 %Thus, according to Fatou's Lemma, 
%\begin{align*}
%\underset{t\to +\infty}{\liminf}\;S(t)e^{-r(1)t}>0. 
%\end{align*}
%
\item Let $\delta\in \big(0,\min\left(r(1)-\tr_\alpha(0), -f'(1)\right)\big)$. 
First, let us note that we can rewrite 
\[Y(t,x)^\alpha=e^{\ln(Y(t,x))-\ln(x)}x^\alpha= x^\alpha e^{\int_{Y(t,x)}^x{-\frac{\alpha}{s}}ds}. \]
Thus, by using expression \eqref{S prime 2}, we get
\begin{align*}
e^{\delta t} \frac{d}{dt}\left(S(t)e^{-r(1)t}\right)&=\int_0^1{n^0\left(Y(t,x)\right)\left(r(x)-r(1)\right)e^{\int_{Y(t,x)}^x{\frac{\tr(s)-r(1)+\delta}{f(s)}ds}}dx}\\
&=\int_0^1{\frac{n^0(Y(t,x))}{Y(t,x)^\alpha}x^\alpha(r(x)-r(1))e^{\int_{Y(t,x)}^x{\frac{\varphi(s)}{f(s)}ds}}dx}, 
\end{align*}
with 
\[\varphi(s):=\tr(s)-r(1)+\delta-\frac{\alpha f(s)}{s}.\]
By hypothesis on $n^0$, $f$ and $r$,  $ \tilde{n^0}:y\mapsto\frac{n^0(y)}{y^\alpha}$ and $\varphi$ are both continuous on $[0,1]$. Moreover, since $\varphi(0)=\tr_\alpha(0)-r(1)+\delta<0$, there exists $\e\in (0,1)$ such that $\varphi(s)<0$ for all $s\in [0, \e]$. Thus, 
\[\bigg\lvert e^{\delta t} \frac{d}{dt}\left(S(t)e^{-r(1)t}\right)\bigg\rvert \leq \lVert \tilde{n^0} \rVert_\infty \int_0^1{\lvert r(x)-r(1) \rvert x^\alpha e^{\int_{\e}^x{\frac{\lvert \varphi(s)\rvert }{f(s)}ds}\ind_{(\e,1)}(x)}dx}. \]
since $\lvert \varphi(1)\rvert =\delta-f'(1)$, Lemma \ref{alpha f} yields 
\[e^{\int_{\e}^x{\frac{\lvert \varphi(s)\rvert }{f(s)}ds}}=\underset{x\to 1}{O}(\lvert x-1\rvert ^{\frac{\delta}{f'(1)}-1}).\]
Since $\lvert r(x)-r(1) \rvert= \underset{x\to 1}{O}(x)$ and $\frac{\delta}{f'(1)}>-1$ (by hypothesis on $\delta$), this proves that this last integral is bounded. 
\end{enumerate}
\item \textbf{Case $\tr_\alpha(0)>r(1)$:}
\begin{enumerate}[(a)]
\item According to \eqref{S uni}, 
\begin{align*}
S(t)e^{-\tr_\alpha(0)t}&=\int_0^1{n^0(Y(t,x))e^{\int_{Y(t,x)}^{x}{\frac{\tr(s)-\tr_\alpha(0)}{f(s)}ds}}dx}\\
&=\int_0^1{\frac{n^0(Y(t,x))}{Y(t,x)^\alpha}Y(t,x)^\alpha e^{\int_{Y(t,x)}^x{\frac{\tr(s)-\tr_\alpha(0)}{f(s)}ds}}dx}.
\end{align*}
By rewriting $Y(t,x)^\alpha=x^\alpha e^{-\int_{Y(t,x)}^x{\frac{\alpha}{s}ds}}$, we get 
\[S(t)e^{-\tr_\alpha(0)t}= \int_0^1{\frac{n^0(Y(t,x))}{Y(t,x)^\alpha} x^\alpha e^{\int_{Y(t,x)}^x{\frac{\tr(s)-\tr_\alpha(0)}{f(s)}-\frac{\alpha}{s}ds}}dx}.  \]
Since $\frac{n^0(Y(t,x))}{Y(t,x)^\alpha} x^\alpha e^{\int_{Y(t,x)}^x{\frac{\tr(s)-\tr_\alpha(0)}{f(s)}-\frac{\alpha}{s}ds}}$ converges pointwise to $C\, x^\alpha  \, e^{\int_0^{x}{\frac{\tr(s)-\tr_\alpha(0)}{f(s)}-\frac{\alpha}{s}ds}}$, which is well-defined, since $s\mapsto \frac{\tr(s)-\tr_\alpha(0)}{f(s)}-\frac{\alpha}{s}$ is continuous at $0$ and positive, we are done.
%Fatou's lemma allows us to conclude that 
%$\underset{t\to +\infty}{\liminf}\;S(t)e^{-\tr(0)t}>0$. 
%
\item Let $\delta\in \big(0, \min(\tr_\alpha(0)-r(1), f'(0))\big)$.
Since $r(1)-\tr_\alpha(0)+\delta<0$, there exists $M\in (0,1)$ such that $r(s)-\tr_\alpha(0)+\delta<0$ for all $s\geq M$. Thus, for all $y\in (0,1)$, $t\geq 0$,  
$$\int_y^{X(t,y)}{\frac{r(s)-\tr_\alpha(0)+\delta}{f(s)}}\leq \int_y^M{\frac{\lvert r(s)-\tr_\alpha(0)+\delta \rvert }{f(s)}ds}\; \ind_{(0,M)}(y).$$ 
Hence, using expression \eqref{S prime unidim}, we get 
\begin{align*}
\bigg\lvert e^{\delta t } \frac{d}{dt}\left( S(t)e^{-\tr(0) t }\right)\bigg\rvert =\bigg\vert \int_0^1{m(y)e^{\int_{y}^{X(t,y)}{\frac{r(s)-\tr(0)+\delta}{f(s)}ds}}}\bigg\rvert &\leq  \int_0^1{\lvert m(y)\rvert e^{\int_{y}^{M}{ \frac{\lvert r(s)-\tr(0)+\delta\rvert}{f(s)}ds}\; \ind_{(0,M)}(y)}dy},
\end{align*}
which is a finite integral, since $\lvert r(0)-\tr_\alpha(0)+\delta \rvert=(1+\alpha)f'(0)+\delta$. Lemma~\ref{alpha f} leads to 
$$e^{\int_{y}^{b}{ \frac{\lvert r(s)-\tr(0)+\delta\rvert}{f(s)}ds}}=\underset{y\to 0}{O}\left(y^{-\frac{\delta}{f'(0)}-\alpha-1}\right).$$ 
The integrability follows from $m(y)=n^0(y)\left(\tr(y)-\tr_\alpha(0)\right)-f(y){n^0}'(y)=\underset{y\to 0}{O}\left(y^{\alpha+1} \right)$ (as seen previously), and $\frac{\delta}{f'(0)}<1$ thanks to our choice for $\delta$. 
\end{enumerate}
\end{itemize}
\item We prove this case with exactly the same arguments that for the case of a unique root which is asymptotically unstable. We therefore apply the proof of \textit{(i)}. 
\end{itemize}
\item[(vii)] In this case, since $f>0$, $X(t,y)\underset{t\to +\infty}{\longrightarrow} +\infty$, for all $y\in \R$. Let us prove that 
\[\underset{t\to+\infty}{\liminf}\; S(t)>0\quad \text{and} \quad \underset{t\to+\infty}{\lim} \; S(t)=0. \]
According to \eqref{S(t) Y(t,x)}, 
\[S(t)=\int_{\supp(n^0)}{n^0(y)e^{\int_y^{X(t,y)}{\frac{r(s)}{f(s)}ds}}dy}. \]
The integrand $n^0(y)e^{\int_y^{X(t,y)}{\frac{r(s)}{f(s)}ds}}$ converges pointwise to $n^0(y)e^{\int_y^{+\infty}{\frac{r(s)}{f(s)}ds}}$, which is well defined (with values in $[0, +\infty]$),  and positive for all $y\in \supp(n^0)$, since $\frac{r}{f}$ is positive.
%Thus, according to Fatou's lemma, 
%\[\underset{t\to+\infty}{\liminf}\; S(t)>0. \]
According to \eqref{S prime 2}, 
\[S'(t)=\int_{\supp(n^0)}{n^0(y)r(X(t,y))e^{\int_y^{X(t,y)}{\frac{r(s)}{f(s)}ds}}dy}. \]
Since $f$ is continuous, positive, and converges to positive constants at $\pm \infty$, $\e:=\underset{s\in \R}{\min}\; f(s)>0$. Thus, for all $y\in \R, t\geq 0$, 
\[\bigg\lvert n^0(y)r(X(t,y))e^{\int_y^{X(t,y)}{\frac{r(s)}{f(s)}ds}} \bigg\rvert\leq \lVert n^0 \rVert_\infty \lVert r \rVert_\infty e^{\frac{\lVert r \rVert_{L^1} }{\e}}<+\infty. \]
Combined with the fact that $r(X(t,y))$ converges to $0$ as $t$ goes to $+\infty$ pointwise, we deduce that $S'$ converges to $0$ by the dominated convergence theorem.
\item[(ix)] Since $f\equiv 0$ on $\E$, $Y(t,x)=x$ for all $(t, x)\in \R_+\times \E$. Thus, according to formula \eqref{Expression S y}, 
\[S(t)=\int_\E{n^0(x)e^{r(x)t}dx}\quad \text{and}\quad S'(t)=\int_\E{n^0(x)r(x)e^{r(x)t}dx}. \]
By Laplace's formula (see \cite{wong2001asymptotic}),
\[S(t)\underset{t\to +\infty}{\sim} \sqrt{2\pi} \left(\sum\limits_{i=1}^p{\frac{n^0(x_i)}{\sqrt{\lvert r''(x_i) \rvert }}} \right)\frac{e^{\bar r t}}{\sqrt{t}}\]
and 
\[ S'(t)\underset{t\to +\infty}{\sim} \sqrt{2\pi} \left(\sum\limits_{i=1}^p{\frac{n^0(x_i)r(x_i)}{\sqrt{\lvert r''(x_i) \rvert }}} \right)\frac{e^{\bar r t}}{\sqrt{t}}=\sqrt{2\pi}\; \bar r \left(\sum\limits_{i=1}^p{\frac{n^0(x_i)}{\sqrt{\lvert r''(x_i) \rvert }}} \right)\frac{e^{\bar r t}}{\sqrt{t}}\underset{t\to +\infty}{\sim} \bar r \; S(t).  \]
Thus, $R(t)=\frac{S'(t)}{S(t)}\underset{t\to +\infty}{\longrightarrow} \bar r $. 
\end{enumerate}
\end{proof}

\subsection{Applications}
\textbf{Summary of the method.}
The method that we propose in order to study the asymptotic behaviour of PDE \eqref{eq intro} can be summarised by the following three steps:
\begin{enumerate}
\item Choose an appropriate family of set $(\O_i)$ which satisfies the assumptions of Proposition \ref{condition O}, and such that we can compute the asymptotic behaviour of the functions $R_i$: a good choice when $f$ has a finite number of roots is to take the interval between the roots, as suggested in Lemma \ref{Oi unidim}. 
\item Use Proposition \ref{properties ODE} in order to determine the limit of $\rho$, and its speed of convergence when possible. 
\item Use the semi-explicit expression of $n$ provided by equation \eqref{n(t,x)}, and eventually Proposition \ref{convergence Radon} to deduce the asymptotic behaviour of $n$. 
\end{enumerate}
In each of the following subsections, we apply the three points detailed in this summary to study the asymptotic behaviour of $n$ in different cases. 

\textbf{Remark regarding the regularity of parameter functions}. As in subsection \ref{subsection R real}, we make the further assumptions that $f\in \mathrm{BV}(\R)$, and that $f$ converges to a non-zero limit at $\pm \infty$. Moreover, we easily check that all the results of this previous section remain true if we assume that $n^0$ is $\mathcal{C}^1$ on each interval between the roots of $f$, and not necessarily on the whole of $\R$. As far as $f$ is concerned, it is enough to assume that it is globally Lipschitz, and $\mathcal{C}^2$ only on a neighbourhood of its roots. It will sometimes be advisable to make these two additional assumptions: we will indicate this at the beginning of each statement whenever this is the case. 

\subsubsection{Case of a unique stable equilibrium}
We start by assuming that $f$ has a unique root (denoted $a$), which is asymptotically stable for the ODE $\dot u =f(u)$. In this case, solutions converge to a weighted Dirac mass at $a$, regardless of the functions $r$ and~$n^0$. The weight in front of the Dirac mass is determined by the value of $r$ at $a$. Note that this result can be generalised to higher dimensions, see Proposition \ref{unique stable equilibrium multidim}. 
\begin{prop}
Let us assume that $f$ has a unique root (denoted $a$), and that $f'(a)<0$. Then, $\rho$ converges to $r(a)$ and $n(t, \cdot) \underset{t\to +\infty}{\rightharpoonup}r(a)\delta_a$.
\label{unique stable}
\end{prop}

\begin{proof}
We apply the three points detailed in the summary:
\begin{enumerate}
\item Let us denote $\O_1:=(-\infty, a)$, $\O_2:=(a, +\infty)$, which satisfy the assumptions of Proposition \ref{prop Oi}, by Lemma \ref{Oi unidim}. By proposition \ref{limit Ri}, $R_1$ and $R_2$ both converge to $r(a)$ (with an exponential speed). 
\item By Proposition \ref{properties ODE}, $\rho$ converges to $r(a)$ with an exponential speed. 
\item According to to the semi-explicit expression \eqref{n(t,x)}, $n(t,x)=n^0(Y(t,x))e^{\int_0^t{\tr(Y(s, x))-\rho(s)ds}}$. Let $\delta>0$. Since $\lVert Y(t,x) \rVert \ul +\infty$ for all $x\in \R^d\backslash \{a\}$, and $n^0$ has a compact support,  there exists $T_0$ such that $n(t,x)=0$ for all $t\geq T_0$, $x\in \R^d \backslash [a-\delta, a+\delta]$. Since $\rho(t)=\int_{\supp(n^0)}{n(t,x)dx}$ converges to $r(a)$, Propositions \ref{convergence Radon} allows us to conclude that $n(t,\cdot)\underset{t\to +\infty}{\rightharpoonup}r(a)\delta_a$.
\end{enumerate}
\end{proof}

\subsubsection{Case of a unique unstable equilibrium}
We now assume that $f$ has a unique root (denoted $a$) which is asymptotically unstable for the ODE $\dot u =f(u)$. Under theses hypotheses, the growth term can counterbalance the advection term: there exist two regimes of convergence, depending on how $r(a)$ and $f'(a)$ compare. 
\begin{prop}
Let us assume that $f$ has a unique root (denoted $a$), and that $f'(a)>0$, Then:
\begin{itemize}
\item If $r(a)<f'(a)$, then $\rho(t)\underset{t\to +\infty}{\longrightarrow}0$ and $n(t, \cdot) \underset{t\to +\infty}{\longrightarrow}0$ in $L^1(\R)$. 
\item  If $r(a)>f'(a)$, and $n^0(a)>0$, then $\rho(t)\underset{t\to +\infty}{\longrightarrow}r(a)-f'(a)$, and $n(t, \cdot) \underset{t\to +\infty}{\longrightarrow}{\bar n}$ in $L^1(\R)$, where 
$$\bar{n}(x):=Ce^{\int_a^x{\frac{\tr(s)-\tr(a)}{f(s)}ds}},$$
with $\tr=r-f'$ and $C$ such that $\int_{\R}{\bar{n}(x)dx}=r(a)-f'(a)$.
\end{itemize}
\label{unique unstable}
\end{prop}
\begin{proof}
We apply the three points detailed in the summary:
\begin{itemize}
\item Let us assume that $r(a)<f'(a)$: 
\begin{enumerate}
\item Let us denote $\O_1:=(-\infty, a)$, $\O_2:=(a, +\infty)$, which satisfy the assumptions of Proposition \ref{prop Oi}, by Lemma \ref{Oi unidim}. Proposition \ref{limit Ri} shows that $R_1$ and $R_2$ both converge to $0$. 
\item By Proposition \ref{properties ODE}, $\rho$ converges to $0$.
\item We immediately deduce from the previous point that $n(t,\cdot)\ul 0$ in $L^1(\R)$, by definition of $\rho$.  
\end{enumerate}
\item Let us assume that $r(a)>f'(a)$: 
\begin{enumerate}
\item With the same choice for $\O_1$ and $\O_2$, Proposition \ref{limit Ri} shows that $R_1$ and $R_2$ both converge to $r(a)-f'(a)$. 
\item By Proposition \ref{properties ODE}, $\rho$ converges to $\tr(a)$ with an exponential speed. 
\item By the semi-explicit expression~\eqref{n(t,x)}, 
\begin{align*}
n(t,x)=n^0(Y(t,x))e^{\int_0^t{\tr(Y(s,x) )-\rho(s)ds}}&=n^0(Y(t,x))e^{\int_0^t{\tr(Y(s, x))-\tr(a)ds}}e^{\int_0^t{\tr(a)-\rho(s)ds}}\\
&=n^0(Y(t,x))e^{\int_{Y(t,x)}^x{\frac{\tr(s)-\tr(a)}{f(s)}ds}}e^{\int_0^t{\tr(a)-\rho(s)ds}}
\end{align*}
(we use the change of variable $s'=Y(s,x)$ in the first integral to get this last expression). 
Thus, $n(t,\cdot)$ converges pointwise to
 $$x\mapsto n^0(a)e^{\int_a^x{\frac{\tr(s)-\tr(a)}{f(s)}ds}}e^{\int_0^{+\infty}{\tr(a)-\rho(s)ds}}, $$
which is well-defined, since $\rho$ converges to $\tr(a)$ with an exponential speed, $f>0$ on $(a, +\infty)$ and $s\mapsto \frac{\tr(s)-\tr(a)}{f(s)}$ is continuous at $a$. 
 
Moreover, since $r(x)\underset{x\to + \infty}{\longrightarrow}0$, and $f$ converges to a positive limit, there exist $M, d>0$ such that $\frac{{r(s)-\tr(a)}}{f(s)}<-d$ for all $s\geq M.$ Thus, for all $t\geq 0$, $x\in (a, +\infty)$, 
\begin{align*}
\int_{Y(t,x)}^x{\frac{\tr(s)-\tr(a)}{f(s)}ds}&\leq \underbrace{\int_a^M{\frac{\lvert \tr(s)-\tr(a)\rvert }{f(s)}ds}}_{:= C_1}+\int_M^x{\frac{\tr(s)-\tr(a)}{f(s)}ds} \; \ind_{(M, +\infty)}(x)\\
&\leq C_1 + \int_M^x{\underbrace{\frac{r(s)-\tr(a)}{f(s)}}_{\leq -d}ds} \; \ind_{(M, +\infty)}(x) + \int_M^x{\frac{f'(s)}{f(s)}ds}\; \ind_{(M, +\infty)}(x)\\
&\leq C_1-d(x-M)\ind_{(M, +\infty)} +\underbrace{\int_M^{+\infty}{\frac{\lvert f'(s) \rvert }{f(s)}ds}}_{:= C_2}, 
\end{align*}
with  $C_1, C_2<+\infty$, by the regularity of $\tr$, $f\in BV(\R)$, and the fact that $f$ converges to a positive constant at infinity. 

By proceeding in the same way for all $x\leq a$, we show that for all $x\in \R$, $t\geq 0$,  
%\[n(t,x)\leq  C e^{-d \lvert x  \rvert \ind_{\R\backslash [-M, M]}(x)}   \]
\[n(t,x)\leq  C e^{-d \lvert x  \rvert}\]
for some constants $ C, d>0$, which ensures, according to the dominated convergence theorem, that 
$t\mapsto n(t, \cdot)$ converges to $x\mapsto n^0(a)e^{\int_a^x{\frac{tr(s)-\tr(a)}{f(s)}ds}}e^{\int_0^{+\infty}{\tr(a)-\rho(s)ds}}$ in $L^1(\R)$. 
%Since  $\rho(t)=\lVert n(t, \cdot) \rVert_{L^1}$ converges to $\tr(a)$, the constant  $e^{\int_0^{+\infty}{\tr(a)-\rho(s)ds}}$ is such the integral of this limit function is equal to $\tr(a)$. 
\end{enumerate}
\end{itemize}
\end{proof}

\subsubsection{Two equilibria}
In this section we assume that $f$ has exactly two roots, $a<b$, which satisfy $f'(a)>0$ and $f'(b)<0$ (hence $f>0$ on $(a,b)$). The case $f'(a)<0$, $f'(b)>0$, $f<0$ on $(a,b)$  is similar.  Depending on the functions $r$ and $n^0$, $n$ will either converge to a function in $L^1$, or converge to a Dirac mass at $b$. We split this result into two propositions: the first one assumes that the support of $n^0$ crosses $a$, which means that $n^0>0$ in a neighbourhood of $a$. The second one assumes that $\supp(n^0)\subset [a, +\infty)$, and we consider the case where $n^0(a)=0$, which leads to other regular functions being reached. 
\begin{prop}
Let us assume that $f$ has exactly two roots, $a<b$, which satisfy $f'(a)>0$, $f'(b)<0$, and that $n^0(a)>0$.  Then: 
\begin{itemize}
\item If $r(b)>r(a)-f'(a)$, then $\rho(t) \ul r(b)$ and $n(t, \cdot)\underset{t\to +\infty}{\rightharpoonup}r(b)\delta_b$. 
\item If $r(b)<r(a)-f'(a)$, then $\rho(t)\ul r(a)-f'(a)$, and $n(t, \cdot)\ul \bar n$ in $L^1(\R)$, where 
\[\bar n(x):= De^{\int_a^x{\frac{\tr(s)-\tr(a)}{f(s)}ds}}\ind_{(-\infty, b)}, \] 
with $\tr=r-f'$, and $D>0$ is such that $\int_{\R}{\bar n(x)dx}=r(a)-f'(a)$.
\end{itemize}
\label{2 equilibria first}
\end{prop}
\begin{proof}
Note that since $n^0$ is assumed to be continuous on $\R$, $n^0>0$ on a neighbourhood of $a$. 
\begin{itemize}
\item Let us assume that $r(b)> r(a)-f'(a)$. We again follow the three points of the method outlined in the beginning of the  subsection. 
\begin{enumerate}
\item Let us denote $\O_1=(-\infty, a)$, $\O_2=(a,b)$, $\O_3=(b, +\infty)$. One easily checks that these sets satisfy the hypotheses of Proposition \ref{prop Oi}, thanks to Lemma \ref{Oi unidim}. According to Proposition \ref{limit Ri}, $R_1$ converges to $\max(0,\tr(a))<r(b)$ and  $R_2$ and $R_3$ both converge to $r(b)$ with an exponential speed. 
\item By Proposition \ref{properties ODE}, $\rho$ converges to $r(b)$ with an exponential speed, and $\rho_1(t)=\int_{-\infty}^a{n(t,x)dx}$ converges to $0$.  
\item  Let $x\in (a,b)$. Using~\eqref{n(t,x)}, we find $n(t,x)=n^0(Y(t,x))e^{\int_0^t{\tr(Y(s,x))-\rho(s))ds}}$, for all $t\geq 0$. Let $K\subset (a, b)$ be a compact set,  $\delta \in \big(0, \frac{1}{2}\left(r(b)-\tr(a)\right)\big)$, and let us denote $d:=r(b)-\tr(a)-2\delta>0$. Since $\rho$  converges to $r(b)$, and $(Y(s,x))_{s\geq 0}$ converges to $a$ uniformly on $K$, there exists $T_0$ such that for all $s\geq T_0$ and  all $x\in K$,  
\[\rho(s)\geq r(b)-\delta \quad \text{and} \quad \tr(Y(s,x))\leq \tr(a)+\delta,  \]  
Thus, 
\[\int_K{n(t,x)dx}\leq \lVert n^0 \rVert_{\infty}  \int_K{e^{\int_0^{T_0}{\tr(Y(s,x))-\rho(s)ds}}dx}\; e^{-d(t-T_0)}\ul 0. \]
Let $K'$ be a compact subset of $(b, +\infty)$. Since $n^0$ has compact support, there exists $T_0$ such that $n^0(Y(t,x))=0$ for all $t\geq T_0$, $x\in K'$. Thus, $t\mapsto \int_{K'}{n(t,x)dx}$ converges to $0$. By Proposition \ref{convergence Radon}, $n(t, \cdot) \underset{t\to +\infty}{\rightharpoonup}r(b)\delta_{b}$. 
\end{enumerate}
\item Let us now assume that $r(b)<r(a)-f(a)$. 
\begin{enumerate}
\item With the same choice for $\O_1, \O_2$ and $\O_3$, Proposition \ref{limit Ri} shows that $R_1$ and $R_2$ converge to $\tr(a)$, and that $R_3$ converges to $r(b)<\tr(a)$. 
\item We then apply Proposition \ref{properties ODE} to infer that $\rho$ converges to $\tr(a)$ with an exponential speed, and $\rho_3(t)=\int_{b}^{+\infty}{n(t,x)dx}$ converges to $0$.  
\item  Let $x\in (-\infty,b)$, $t\geq 0$.  By the semi-explicit expression \eqref{n(t,x)},  
\begin{align*}
n(t,x)&=n^0(Y(t,x))e^{\int_0^t{\tr(Y(s,x))-\rho(s))ds}}\\
&=n^0(Y(t,x))e^{\int_{Y(t,x)}^x{\frac{\tr(s)-\tr(a)}{f(s)}ds}}e^{\int_0^t{\tr(a)-\rho(s)ds}},
\end{align*}
where we used the change of variable `$s'=Y(t,x)$'. The latter function converges pointwise to
 $$n^0(a)e^{\int_a^x{\frac{\tr(s)-\tr(a)}{f(s)}ds}}e^{\int_0^{+\infty}{\tr(a)-\rho(s)ds}}.$$
As for the case of a unique unstable equilibrium (proof of Proposition \ref{unique unstable}) one can find $C,d \geq 0$ such that $n(t,x)\leq C e^{-d\lvert x \rvert }$ for all $x\leq a$. Moreover, for all $x\in (a,b)$, 
$$n(t,x)\leq \lVert n^0 \rVert_{\infty}\; e^{\int_0^{+\infty}{\lvert \tr(a)-\rho(s) \rvert ds}}\;  e^{\int_a^x{\frac{\tr(s)-\tr(a)}{f(s)}\ind_{\{\tr(s)>\tr(a)\}}(s)ds}},$$
which provides an $L^1$-domination, since $e^{\int_a^x{\frac{\tr(s)-\tr(a)}{f(s)}}}=\underset{x\to +\infty}{O}\left(\lvert x-b \rvert^{\frac{\tr(a)-r(b)}{-f'(b)}-1} \right)$, thanks to Lemma~\ref{alpha f}. By the dominated convergence theorem, combined with the fact that $\rho$ converges to $\tr(a)$ and $\rho_3$ converges to $0$, this ensures that $n(t,\cdot)$ converges to the expected limit. 
\end{enumerate}
\end{itemize}
\end{proof}
In the following proposition, we assume that $n^0$ is $\mathcal{C}^1$ on $(a,b)$ and on $(b, +\infty)$, and not necessarily on the whole of $\R$. 
\begin{prop}
Let us assume that $f$ has exactly two roots, $a<b$, which satisfy $f'(a)>0$, $f'(b)<0$, and that $\supp(n^0)\subset [a,+\infty)$. We distinguish between several cases:
\begin{itemize}
\item If $n^0(a)>0$, then
\begin{itemize}
\item If $r(b)>r(a)-f'(a)$, then $\rho(t)\ul r(b)$, and $n(t, \cdot)\ui{\rightharpoonup} r(b)\delta_b$.
\item If $r(b)<r(a)-f'(a)$, then $\rho(t)\ul r(a)-f'(a)$, and $n(t, \cdot)\ul \bar n_0$ in $L^1(\R)$, where 
\[\bar n_0(x):= D_0e^{\int_a^x{\frac{\tr(s)-\tr(a)}{f(s)}ds}}\ind_{(a, b)}, \] 
with $\tr=r-f'$, and $D_0>0$ is such that $\int_{\R}{\bar n_0(x)dx}=r(a)-f'(a)$.
\end{itemize}
\item If $n^0(a)=0$, and if there exist $C, \alpha>0$ such that ${n^0}'(y)=C\alpha (y-a)^{\alpha-1}+\underset{y\to a^+}{O}\left((y-a)^\alpha\right)$, then
\begin{itemize}
\item If $r(b)>r(a)-(1+\alpha)f'(a)$, then $\rho(t)\ul r(b)$, and $n(t, \cdot)\ui{\rightharpoonup} r(b)\delta_b$.
\item If $r(b)<r(a)-(1+\alpha)f'(a)$, then $\rho(t)\ul r(a)-(1+\alpha)f'(a)$, and $n(t, \cdot)\ul \bar n_\alpha$ in $L^1(\R)$, where 
\[\bar n_\alpha(x):= D_\alpha (x-a)^\alpha e^{\int_a^x{\frac{\tr(s)-\tr_\alpha(a)}{f(s)}-\frac{\alpha}{s-a}ds}}\ind_{(a, b)}, \] 
where $\tr=r-f'$, $\tr_\alpha= r-(1+\alpha)f'$, and $D_\alpha>0$ is such that $\int_{\R}{\bar n_\alpha(x)dx}=r(a)-(1+\alpha)f'(a)$. 
\end{itemize}
\item If $n^0(a)=0$, and if there exists $\e>0$ such that $n^0(y)=0$ for all $y\in [a, a+\e]$, then $\rho(t) \underset{t\to +\infty}{\longrightarrow}r(b)$, and $n(t, \cdot)\underset{t\to +\infty}{\rightharpoonup}r(b)\delta_b$. 
\end{itemize}
\label{2 equilibria second}
\end{prop}
\begin{proof}
Since the proof of this proposition is very similar to the one of Proposition \ref{2 equilibria first}, we do not write it in full detail, but we  simply underline the points that must be adapted. 
\begin{itemize}
\item In the case where $n^0(a)>0$, and $\supp(n^0)\subset [a, +\infty)$, the proof is the same, but by considering only the two sets $\O_2=(a,b)$, and $\O_3=(b, +\infty)$, and not $\O_1=(-\infty, a)$. We easily check using Lemma \ref{Oi unidim} that $(\O_2, \O_3)$ satisfy the hypotheses of Lemma \ref{Oi unidim}, since $\supp(n^0)\cap \O_1=\emptyset$
\item The case where $n^0(a)=0$ and the hypothesis on ${n^0}'$ holds is quite similar, except that Proposition \ref{limit Ri} now shows that $R_2$ converges to $\max \left( \tr_\alpha(a), r(b)\right)$, with an exponential speed (if $\tr_\alpha(a)\neq r(b)$). Thus, we treat the case $r(b)>\tr_\alpha(a)$ in exactly the same way; the case $r(b)<r(a)-(1+\alpha)$ is a little more intricate: recalling that for all $x\in (a,b)$, $t\geq 0$, 
$$n(t,x)=n^0(Y(t,x))e^{\int_{Y(t,x)}^x{\frac{\tr(s)-\tr_\alpha(a)}{f(s)}ds}} e^{\int_0^t{\tr_\alpha(a)-\rho(s) ds}} \quad \text{and} \quad \left(Y(t,x)-a\right)^\alpha=(x-a)^\alpha e^{-\int_{Y(t,x)}^x{\frac{\alpha}{s-a}}ds}, $$
one notes that 
$$n(t,x)= \frac{n^0(Y(t,x))}{(Y(t,x)-a)^\alpha} e^{\int_0^t{\tr_\alpha(a)- \rho(s)ds}} (x-a)^\alpha  e^{\int_{Y(t,x)}^x{\frac{\tr(s)- \tr_\alpha(a)}{f(s)}-\frac{\alpha}{s-a}ds}}$$
converges pointwise to 
\[C  e^{\int_0^{+\infty}{\tr_\alpha(a)- \rho(s)ds}} (x-a)^\alpha  e^{\int_{a}^x{\frac{\tr(s)- \tr_\alpha(a)}{f(s)}-\frac{\alpha}{s-a}ds}}, \]
which is well-defined, since $\rho$ converges to $\tr_\alpha(a)$ with an exponential speed, and $s\mapsto \frac{\tr(s)-\tr_\alpha(s)}{f(s)}-\frac{\alpha}{s-a}$ is continuous on $a$, as seen in the proof of Proposition \ref{limit Ri}. Moreover, 
\[x\mapsto \lVert \frac{n^0(\cdot)}{(\cdot-a)^\alpha} \rVert_{\infty}e^{\int_0^{+\infty}{\lvert \tr_\alpha(s)-\rho(s) \rvert ds }}\,e^{\int_a^x{\frac{\varphi(s)}{f(s)}\ind_{\{\varphi(s)\geq 0\}}ds}}, \]
with $\varphi(s)=\tr(s)-\tr_\alpha(s)-\alpha \frac{f(s)}{s-a}$ is clearly a domination of $n$, and is in $L^1$, since $\varphi(b)=r(b)-\tr_\alpha(a)-f'(b)$, which implies by Lemma \ref{alpha f} that $e^{\int_a^x{\frac{\varphi(s)}{f(s)}ds}}=\underset{x\to b^-}O\left((b-x)^{\frac{r(b)-\tr_\alpha(a)}{f'(b)}-1}\right)$, with $\frac{r(b)-\tr_\alpha(a)}{f'(b)}>0$. 
\item This last point is the simplest, and is in fact analogous to the case of a single equilibrium point. According to Proposition \ref{limit Ri}, $R_2$ and $R_3$ converge to $r(b)$: we deduce the result by following the steps of Proposition \ref{unique stable}. 
\end{itemize}
\end{proof}
\begin{rem}
Since Proposition \ref{2 equilibria second} provides a completely explicit expression for the limit functions $\bn_\alpha$, $\alpha\geq 0$, one can easily determine their asymptotic behaviour at the boundary of the segment $(a,b)$. 
Since for all $\alpha>0$, $x\in (a,b)$, 
 $$\bar n_\alpha(x)= D_\alpha (x-a)^\alpha e^{\int_a^x{\frac{\tr(s)-\tr_\alpha(s)}{f(s)}-\frac{\alpha}{s-a}ds}}, $$ 
and $s\mapsto \frac{\tr(s)-\tr_\alpha(a)}{f(s)}-\frac{\alpha}{s}$ is continuous on $[a,b)$, it is clear that $\bar n_\alpha(x)=\underset{x\to a^+}{\Theta}\left((x-a)^\alpha\right)$. In particular, $\bar n_\alpha$ can be extended by continuity at $0$, with $\bn_\alpha(a)=0$ if $\alpha>0$, and $\bn_0(a)\in (0, +\infty)$. 

Moreover, since $\tr(b)-\tr_\alpha(a)-\frac{\alpha f(b)}{b-a}= r(b)-\tr_\alpha(a)-f(b)$, Lemma \ref{alpha f} ensures that $\bn_\alpha(x)=\underset{x\to b^-}{\Theta}\left((b-x)^{\frac{r(b)-\tr_\alpha(a)}{f'(b)}-1}\right)$. In particular 
\begin{align*}
\underset{x\to b^-}{\lim} \bn_\alpha(x)= \begin{cases} 0 \; &\text{ if } \; \tr(b)<\tr_\alpha(a)\\
l>0 \;  & \text{ if } \; \tr(b)=\tr_\alpha(a)\\
+\infty \; & \text{ if } \;  \tr(b)>\tr_\alpha(a) 
\end{cases}. 
\end{align*}
These different cases are illustrated by Figure \ref{continuous profiles}. 
\end{rem}
The case where $f$ has two roots $a<b$ with $f'(a)<0$ and $f'(b)>0$ is symmetric to the cases here, and thus lead to the same results, by switching $a$ and $b$ in the Propositions. 
\begin{figure}
\centering
\includegraphics[width=12 cm]{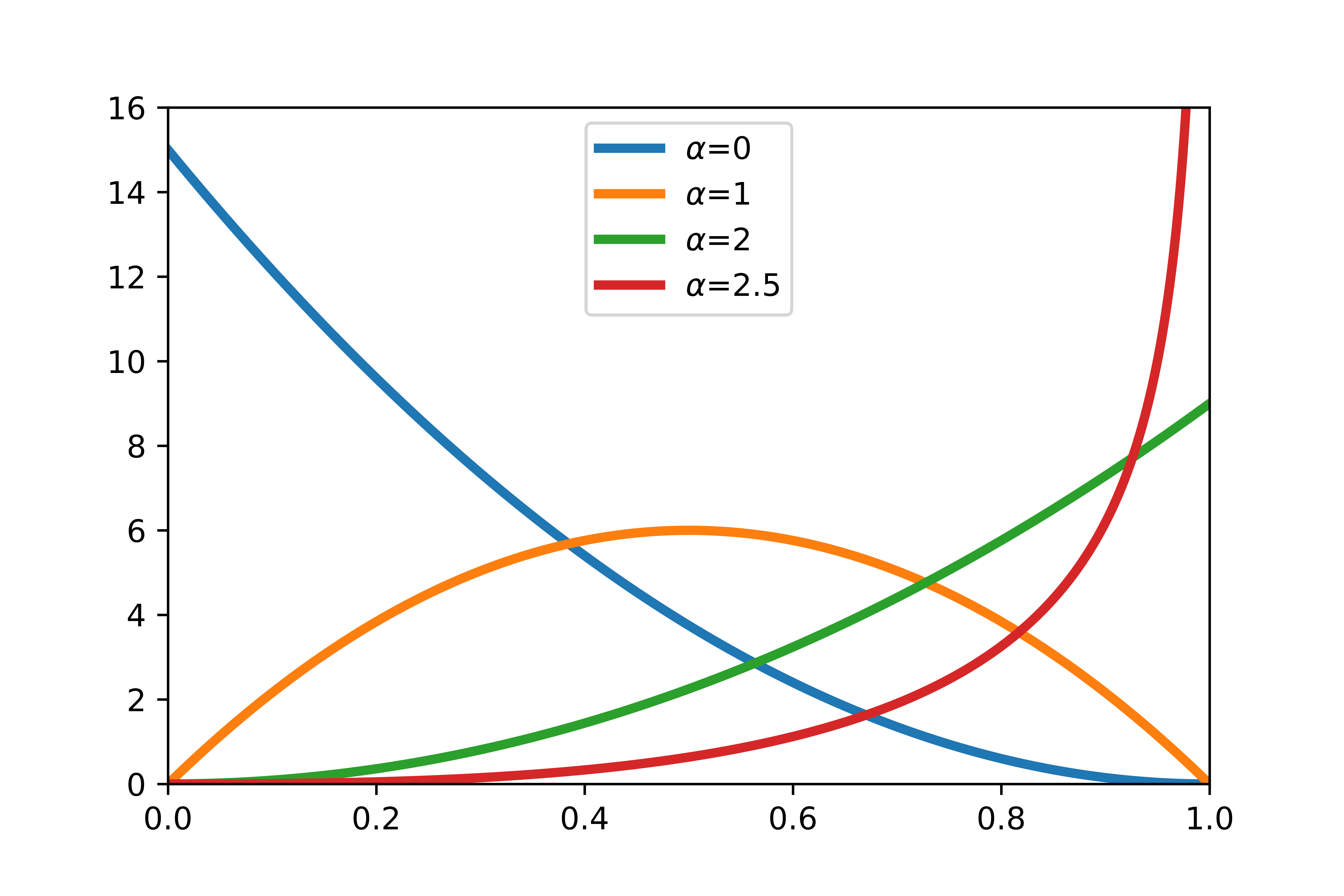}
\caption{Continuous limit functions $\bar n_\alpha$, for different values of $\alpha>0$, as defined in Proposition \ref{2 equilibria second}.  In this example, we have chosen $f(x)=x(1-x)$, and $r(x)=b-ax$ (with $b=6, a=4$). With this choice, we easily compute that, for all $\alpha \in [0, a-1)$, and all $x\in (0, 1)$ $\bar n_\alpha(x)=D_\alpha x^\alpha(1-x)^{a-\alpha-2} $, for the appropriate constant $D_\alpha$. This illustrates the variety of limit functions that can be reached depending on the initial condition, as detailed in Proposition \ref{2 equilibria second}. \\}
\label{continuous profiles}
\end{figure}

\subsection{More than two equilibria}

In this subsection, we deal with the cases where $f$ has more than two equilibria. As evidenced by the previous result, listing all possible scenarios when there are two roots already is cumbersome: this is why we will not do so in a more general case, and will focus on the case where $n^0$ is positive on the neighbourhood of the unstable equilibrium points. The other cases can of course be treated as seen above, keeping in mind that this changes the value of the limits reached by the $R$ functions.
\begin{prop}
Let us assume that $f$ has a finite number of roots, which are all hyperbolic equilibrium points for the ODE $\dot u=f(u)$, \textit{i.e.} $f'$ has a sign at each root of $f$,  and let us denote $x_{u}^1, ... , x_{u}^p$ the asymptotically unstable equilibria,  and $x_{s}^1, ..., x_{s}^1$ the asymptotically stable one. Moreover, let us denote  $M_u:=\max\{r(x_{u}^1)-f'(x_{u}^1),... r(x_{u}^p)-f'(x_{u}^p)\}$, and $M_s:= \max\{r(x_{s}^1), ..., r(x_{s}^m)\}$, and let us assume that these two maxima are both reached at a unique point. Lastly, let us assume that $n^0(x_u^i)>0$ for all $i\in \{1,..., p\}.$ 
\begin{itemize}
\item If $M_s>M_u$, then $\rho(t)\ul M_s$, and $n(t, \cdot) \underset{t\to+\infty}{\rightharpoonup}M_s \delta_{x^i_s}$, with $x_s^i$ the unique stable equilibria such that $M_s=r(x_s^i)$. 
\item %If $M_s<M_u$, then $\rho(t)\ul M_u$, and  $n$ converges in $L^1$ to a function determined by the unique  $x_{u}^i$ such that  $r(x_u^i)-f'(x_u^i)=M_u$. 
If $M_s<M_u$, then $\rho(t)\ul M_u$, and  $n(t, \cdot )\underset{t\to +\infty}{\longrightarrow} \bar n_{i^*}$  in $L^1$, where 
\[\bar n_{i^*}(x)= C_{i^*}e^{\int_{x_u^{i^*}}^x{\frac{\tr(s)-\tr(x_u^{i^*})}{f(s)}ds}}\ind_{I_{i^*}}(x), \]
with $\tr=r-f'$, $i^*$ the unique integer of $\{1,..., p\}$ such that $\tr(x_u^{i^*})=M_u$, $I_{i^*}$ the open interval delimited by the two stable equilibria which enclose $x_u^{i^*}$ (or $ -\infty$ or $+\infty$ if $x_u^{i^*}$ is the smallest or the greatest root of $f$), and $C_{i^*}$ a positive constant such that $\int_{I_{i^*}}{\bar n_{i^*}(x)dx}=M_u$. 
\end{itemize}
\label{multi equilibria}
\end{prop}
\begin{proof}
The proof of this proposition is in similar to that of Proposition \ref{2 equilibria first}: we denote $\O_0, ..., \O_{p+m}$, the intervals between each roots of $f$, which satisfy the hypotheses of Proposition \ref{prop Oi}, according to Lemma \ref{Oi unidim}, and, using Proposition \ref{limit Ri}, we are able to compute the limit of the function $R_i$, for all $i\in \{0, \ldots, p+m\}$, and thus determine the long-time behaviour of $\rho$ by Proposition \ref{properties ODE}. We conclude by using the semi-explicit expression~\eqref{n(t,x)} for $n$. 
\end{proof}

This method also allows to deal with the case where $f\equiv 0$ on a whole segment: we do not return to the case $f\equiv 0$ on $\R$, which has already been studied in \cite{perthame2006transport} and \cite{lorenzi2020asymptotic}, but we consider the case where $f\equiv 0$ on an interval, and then becomes positive. 

To make the assumption of the following proposition possible, we assume that $f$ is $\mathcal{C}^2$ on $(-\infty, a)$ and on $(a, +\infty)$, but not necessarily on the whole of $\R$. 
\begin{prop}
Let us assume that there exists $a\in \R$ such that $f\equiv 0$ on  $(-\infty, a)$, $f>0$ on $(a,+\infty)$, $f'(a^+ )>0$, and that $\supp(n^0)=[s^-, s^+]$, with $s^-<a<s^+$. Then, 
\begin{itemize}
\item If there exists a unique $x_M\in (s^-, a)$ such that $r(x_M)=\underset{x\in [s^-, a]}{\max}\; r(x)$, and $f''(x_M)<0$, then $\rho$ converges to $r(x_M)$, and $n(t, \cdot)\underset{t\to +\infty}{\rightharpoonup}r(x_M)\delta_{x_M}$. 
\item If $r_{\mid_{[s^-, a]}}$ reaches its maximum at $a$ (and only at $a$), then $\rho$ converges to $r(a)$, and $n(t, \cdot)\underset{t\to +\infty}{\rightharpoonup}r(a)\delta_{a}$. 
\end{itemize}
\label{f equiv 0 on a segment}
\end{prop}
\begin{proof}
\begin{itemize}
\item
\begin{enumerate}
\item Let us denote $\O_1:=(s^-, a)$, $\O_2:=(a, +\infty)$, which satisfy the hypothesis of Proposition~\ref{prop Oi}, according to Lemma \ref{Oi unidim}. By Proposition \ref{limit Ri}, $R_1$   converges to $r(x_M)$ and $R_2$ converges to $r(x_M)-f'(x_M)$. 
\item From Proposition \ref{properties ODE}, $\rho$ and $\rho_1$ converge
 to $r(x_M)$, and $\rho_2$ converges to $0$. 
\item Let $K\subset [s^-, a]$ be a compact set that does not contain $x_M$. Thanks to the semi-explicit expression \eqref{n(t,x)}, and using the fact that $f'(x)=0$ and $Y(t,x)=x$ for all $x\in K$ and all $t\geq 0$, 
\[n(t,x)=n^0(x)e^{\int_0^t{r(x)-\rho(s)ds}}\leq n^0(x)e^{\int_0^t{r_K-\rho(s)ds}},\]
with $r_K:=\underset{x\in K}{\max}\; r(x)<r(x_M).$

Thus, $$\int_K{n(t,x)dx}\leq \int_K{n^0(x)dx}e^{\int_0^t{r_M-\rho(s)ds}},$$
which converges to $0$, since $r_M-\rho(s)$ is negative for any $s$ large enough. This proves the result thanks to Proposition \ref{convergence Radon}. 
\end{enumerate}
\item 
\begin{enumerate}
\item Here we have to make a slightly more subtle choice of subsets than usual. Let $\e>0$, and let us denote $\O_1^\e:=(s^-, a-2\e)$, $\O_2^\e:=(a-2\e, a-\e)$, $\O_3^\e:=(a-\e, a)$, $\O_4:=(a, +\infty)$. We easily check that these four sets satisfy the hypotheses of Proposition \ref{prop Oi}. Moreover, since $f\equiv 0$ on $[s^-, a]$ for all $i\in \{1,2,3\}$,
\[R_i^\e(t)=\frac{\int_{\O_i^\e}{r(x)e^{r(x)t}dx}}{\int_{\O_i^\e}{e^{r(x)t}dx}}. \]
Thus, for all $t\geq 0$, $i\in \{1,2,3\}$, 
\[\underset{x\in \bar{\O}_i^{\e}}{\min}\; r(x)\leq R_i^{\e}(t)\leq \underset{x\in \bar{\O}_i^{\e}}{\max}\; r(x).\]
In particular, 
\[\bar{R_1^\e}\leq \underset{x\in [s^-, a-2\e] }{\max}\; r(x)\quad \text{and}\quad  \underline{R_3^\e} \geq \underset{x\in [a-\e, a] }{\min}\; r(x). \]
Finally, Proposition \ref{limit Ri} shows that $R_4$ converges to $r(a)-f'(a^+)$. 
\item Since $r$ reaches its unique maximum at $a$, for any $\e$ small enough, we get
\[\underline{R_3^\e}> \bar{R_1^\e}\quad \text{and} \quad \underline{R_3^\e}> \underset{t\to +\infty}{\lim}\; R_4(t). \]
Thus, according to Proposition \ref{properties ODE}, $\rho^\e_1$ and $\rho_4$ converge to $0$, for all $\e>0$. The choice of $\e$ being arbitrary, it also proves that $\rho_2^\e$ converges to $0$. Thus, $\bar \rho = \bar{\rho_3^\e}$, and $\underline{\rho}= \underline{\rho_3^\e}$, for all $\e>0$. Since for all $t\geq 0$ 
\[\underset{x\in [a-\e, a]}{\min}\; r(x)\leq R_3^\e(t)\leq r(a), \] 
we prove that $\rho$ converges to $r(a)$ by making $\e$ tend to $0$.  
\item We have proved that $t\mapsto \int_{s^-}^a{n(t,x)dx}$ converges to $r(a)$, that $t\mapsto \int_a^{+\infty}{n(t,x)dx}$ converges to~$0$ and that for all $\e>0$, $\int_{s^-}^{a-\e}{n(t,x)dx}$ converges to $0$. The hypotheses of Proposition \ref{convergence Radon} are therefore met, which concludes the proof. 
\end{enumerate}
\end{itemize}
\end{proof}
Note that the methods of Propositions~\ref{multi equilibria} and \ref{f equiv 0 on a segment} can be coupled to treat more complex cases, where, for example, $f\equiv 0$ on several disjoint segments.  

\section{Some results in higher dimensions}
\label{section higher dimension}
As seen in the previous sections, our entire method is based on the computation of the limit of the $R_i$ functions defined in Section \ref{section resolution}.  Unfortunately, these computations seem out of reach in the multidimensional case $\R^d$, $d\geq 2$.

In this section, we nevertheless address the question of the possible convergence to smooth or singular measures in higher dimensions in some specific simple cases. We first analyse how the solution support evolves over time. This allows us to conclude that that the solution converges to a Dirac mass in the case of a unique equilibrium which is asymptotically stable for the ODE $\dot u =f(u)$, and provide hypotheses under which the solution cannot converge to a smooth function. We then characterise which stationary measures may or may not be limits for solutions of \eqref{eq intro}, before providing a criterion ensuring the existence of continuous stationary solutions. 

\subsection{Limit support}

\begin{definition}[Limit support]
We define the \textbf{limit support} of $n$ as: 
\[\sigma_{\infty}=\underset{t\geq0}{\bigcap}\overline{\underset{s\geq t}{\bigcup}\text{supp}\left(n(s, \cdot )\right)}.\]
\end{definition}
Recalling the semi-explicit expression \eqref{n(t,x)}, 
\[n(t,x)=n^0(Y(t,x))e^{\int_0^t{(r-\dive f )(Y(s,x))-\rho(s)ds}}, \]
and that for all $t\geq 0$, $\supp\left(n^0(Y(t, \cdot))\right)=X(t, \supp(n^0))$, we get
\begin{align}
\sigma_\infty = \underset{t\geq0}{\bigcap}\overline{\underset{s\geq t}{\bigcup}\supp \left( n^0\left(Y( s, \cdot)\right)\right)}=  \underset{t\geq 0}{\bigcap}\overline{\underset{s\geq t}{\bigcup}X\left(s, \textrm{supp}\left( n^0 \right)\right)}.
\label{limit support inclusion}
\end{align}
In the cases where we are able to determine the latter set, we gather information about possible limits for $n$.
\begin{lem}
If the limit support of $n$ is of measure zero, then $n$ does not converge (weakly) to a non-zero function in $L^1(\R^d)$. 
\label{lemma limit support}
\end{lem}
\begin{proof}
%Let us argue by contradiction. Let us assume that the limit support of $n$ is of measure zero, and that it converges weakly to $\bar n \in L^1(\R^d)$, $\bar n \not\equiv 0$. 
Let us argue by contradiction. By denoting $\nu$ the Lebesgue measure, let us assume that $\nu(\sigma_\infty)=0$, and that $n$ converges weakly to $\bar n \in L^1(\R^d)$, $\bar n \not\equiv 0$. 
Since $\underset{t\to +\infty}{\limsup}\,\supp \left(n\left(t,\cdot\right)\right)= \underset{t\geq0}{\bigcap}\underset{s\geq t}{\bigcup}\text{supp}\left(n(s, \cdot )\right)\subset \sigma_\infty$, 
\[\underset{t\to +\infty}{\limsup}\,\nu \left(\supp(n(t, \cdot))\right)\leq \nu \Big(\underset{t\to +\infty}{\limsup}\,\supp \left( n(t,\cdot\right))\Big)\leq \nu(\sigma_\infty)=0, \]
which contradicts the initial hypothesis. 
%Let $\e>0$.  By definition of the limit support, there exist $T_\e>0, K_\e\subset\R^d$ such that for all $t\geq T_\e$, $\supp(n(t, \cdot))\subset K_\e$ and $\nu(K_e)\leq \e$, where $\nu$ denotes the Lebesgue measure. According to the definition of weak convergence, for all $\varphi \in \mathcal{C}_c(\R^d)$, 
%$$\int_{\R^d}{\varphi(x)n(t,x)dx}\ul \int_{\R^d}{\varphi(x)\bar n(x)dx}.$$
%In particular, for all $\varphi \in \mathcal{C}_c(\R^d\backslash K_\e)$, $\int_{\R^d}{\varphi(x)\bar n(x)dx}=0$, which ensures that $\bar n \equiv 0$ on $\R^d\backslash K_\e$.  Since this result holds for all $\e>0$, it proves that $\bar n \equiv 0$ almost everywhere, which contradicts the initial hypotheses. 
\end{proof}

\begin{prop} Let us assume that $f$ has a unique root, denoted $\bar x$, which is \textbf{globally asymptotically stable} for the ODE $\dot u =f(u)$ over $\R^d$, and that the set $\underset{t\geq 0}{\bigcup} X(t, \supp(n^0))$ is bounded. 
Then, $n(t, \cdot) \underset{t\to +\infty}{\rightharpoonup}r(\bar x)\delta_{\bar x}$, and $\rho$ converges to $r(\bar x)$. 
\label{unique stable equilibrium multidim}
\end{prop}
\begin{proof}
Since the support of $n^0$ is compact, and $\bar x$ is globally asymptotically stable, we easily check, according to \eqref{limit support inclusion}, that $\sigma_\infty= \{\bar x\}$. By Lemma \ref{convergence Radon}, it is hence enough to prove that $\rho$ converges to $r(\bar x)$. As seen in the proof of Lemma \ref{bounds rho}, $\rho$ satisfies, for all $t\geq 0$, 
\[\dot \rho(t)=\int_{\R^d}{\big(r(x)-\rho (t)\big)n(t,x)dx}= \int_{\supp \left(n(t, \cdot)\right)}{\big( r(x)-\rho (t)\big) n(t,x)dx}.\]
Let $\e>0$. Since $\sigma_\infty=\{\bar x\}$ is the intersection of compact decreasing sets, there exists $T_\e>0$ such that, for all $t\geq T_\e$, $\supp(n(t, \cdot))\subset B(\bar x, \e)$. Thus, by denoting 
$$r_m^\e:=\underset{x\in B(\bar x, \e)}{\min}\,r(x)\quad \textrm{and}\quad r_M^\e:=\underset{x\in B(\bar x, \e)}{\max}\,r(x), $$
we get, for all $t\geq T_\e$, 
$$(r_m^\e-\rho(t))\rho(t)\leq \dot \rho(t)\leq (r_M^\e-\rho(t))\rho(t), $$
which ensures that 
\[\underset{t\to +\infty}{\liminf}\, \rho(t)\geq r_m^\e\quad \textrm{and} \quad \underset{t\to +\infty}{\limsup}\, \rho(t)\leq r_M^\e. \]
Since these inequalities hold for any $\e>0$, and $r_m^\e$ and $r_M^\e$ both converge to $r(\bar x)$ when $\e$ goes to $0$, it concludes the proof. 
\end{proof}
Because of the diversity of possible behaviours of ODE systems, it is difficult to compute the limit support for a given $f$, unless very strong assumptions are made about the ODE $\dot u=f(u)$. This is what we do in the following proposition, motivated by a family of ODE systems commonly used in systems biology. 

We say that the two-dimensional system 
\begin{align}
\begin{cases}
\dot x_1=f_1(x_1, x_2)\\
\dot x_2=f_2(x_1, x_2)
\end{cases}
\label{ODE coop/comp}
\end{align}
is \textbf{competitive} if $\partial_2 f_1 \leq 0$ and  $\partial_1 f_2 \leq 0$, and \textbf{cooperative} if $\partial_2 f_1 \geq 0$ and $\partial_1 f_2 \geq 0$.  For instance, such systems are commonly used to model the interactions between two proteins in the context of cell differentiation \cite{guantes2008multistable, thomas2002laws, gardner2000construction, jia2017operating}, and are known to have an interesting property: trajectories either go to $+\infty$, or converge \cite{hirsch1982systems}, \textit{i.e.} for all $x\in \R^d$,
\begin{align}
\lVert Y(t,x)\rVert \underset{t\to +\infty}{\not\longrightarrow} +\infty \quad \Rightarrow\quad t\mapsto Y(t,x) \textrm{ converges }. 
\label{infty or converges}
\end{align}
Note that if the ODE \eqref{ODE coop/comp} is competitive (or cooperative), then the reverse ODE $\dot u=-f(u)$ is cooperative (or competitive). This motivates the hypothesis of the following proposition. Before giving its statement, we recall that if $\bar x$ is a root of $f$, $\bar x$ is called a \textbf{hyperbolic equilibrium} if all the eigenvalues of $\textrm{Jac } f(\bar x)$ have a non-zero real part, and is called a \textbf{repellor} if all these eigenvalues have a positive real part. Lastly, we recall that the \textbf{unstable set} of $\bar x$ is defined by $\{x\in \R^d: Y(t,x)\ul \bar x\}$.

\begin{prop}
Let us assume that $f$ has a finite number of roots, and is such that identity \eqref{infty or converges} holds.
Then, the limit support of $n$ is included in the closure of the union of the unstable sets of the roots of $f$, \textit{i.e.} by denoting $\bar x_1, .... \bar x_N$ the roots of $f$, 
\[\sigma_\infty\subset \overline{\underset{1\leq i\leq N}\bigcup\left\{x\in \R^d: Y(t,x)\ul \bar x_i\right\}}. \]
Moreover, if all the roots of $f$ are hyperpolic points, and if none of them is a repellor, then the limit support of $n$ is of measure $0$. In particular, $n$ does not converge (weakly) to a function in $L^1$. 
\label{non conv in L1}
\end{prop}
\begin{proof}
The inclusion is clear: by hypothesis for all $x\in \R^d$ such that $t\mapsto Y(t,x)$ does not converge, $t\mapsto \lVert Y(t,x)\rVert$ goes to $+\infty$, and since the support of $n^0$ is bounded,  the points of the limit support are necessary in the unstable set of one of the equilibria. The second part of the proposition is a consequence of the stable manifold theorem \cite{perko2013differential}, which ensures that the unstable set of an equilibrium which is not a repellor is a smooth manifold of dimension at most $d-1$, hence a set of measure zero. We conclude with Lemma \ref{lemma limit support}. 
\end{proof}

\subsection{Stationary solutions}
In this subsection, we define the stationary solution in the weak sense, which allows to include measures. As seen in the previous section, under appropriate hypotheses on $f$, the presence of a repellor is necessary to hope for solutions which converge to smooth functions. In this section, we prove that, under appropriate hypotheses, the presence of a repellor ensures the existence of smooth stationary solutions. 
\begin{definition}[Weak stationary solution]
Let $\mu$ be a finite positive Radon measure. We say that $\mu$ is a \textbf{weak stationary solution} of equation \eqref{eq intro} if it satisfies 
\begin{align}
\forall \varphi \in \mathcal{C}^1_c\left(\R^d\right),\quad  \int_{\R^d}{ \left( f(x).\nabla \varphi(x)+ (r(x)-\mu(\R^d))\varphi(x)\right) d\mu(x) }=0.
\label{weak solution}
\end{align}
\end{definition}
\begin{rem}
  If $\bar x$ is a root of $f$, let us note that $r(\bar x)\delta_{\bar x}$ is a weak stationary solution of \eqref{eq intro}. 
\end{rem}
The following proposition shows, as we might expect, that convergent solutions of \eqref{eq intro} (in the weak sense) necessarily converge to a weak stationary solution.
\begin{prop}
Let us assume that $r\in \mathcal{C}_0(\R^d)$, and let $n(t,\cdot)$ be a solution of \eqref{eq intro} which converges in the weak sense in the space of Radon measure. Then its limit is a weak stationary solution of equation \eqref{eq intro}. 
\label{conv to stationary sol}
\end{prop}
\begin{proof}
We let $\mu$ be the limit of $n(t,\cdot)$.
Let us first prove that, under these conditions, $\rho(t)=\int_{\R^d}{n(t,x)dx}$ converges when $t$ goes to $+\infty$.

Let us denote $\psi(t):=\int_{\R^d}{r(x) n(t,x)dx}$, which is non-negative, according to the non-negativity of $r$ and $n$, and converges to $\bar \psi :=\int_{\R^d}{r(x)d\mu(x)dx}$, by definition of the weak convergence, and since $r\in \mathcal{C}_0(\R^d)$. Let us assume that $\bar \psi>0$. 
Let $\e\in (0, \bar \psi)$. Since $\psi$ converges to $\bar \psi$, and since $\rho$ satisfies the ODE
\begin{align*}
\dot \rho(t)=\psi(t)-\rho(t)^2, 
\end{align*}
 there exists $T_\e>0$ such that for all $t\geq T_\e$, 
\[\bar \psi -\e -\rho(t)^2\leq \dot \rho(t)\leq \bar \psi +\e -\rho(t)^2.\]
In other words, 
$\rho$ is a super-solution of $\dot u = \bar \psi -\e -u^2$, and a sub-solution of $\dot u = \bar \psi +\e -u^2$. Since the solutions of these equations converge to $\sqrt{\bar \psi-\e}$ and $\sqrt{\bar \psi+\e}$ respectively, 
\[\underset{t\to +\infty}{\liminf}\; \rho(t) \geq \sqrt{\bar \psi -\e} \quad \text{and} \quad \underset{t\to +\infty}{\limsup}\; \rho(t) \leq \sqrt{\bar \psi +\e}. \]
Since these inequalities hold for any $\e\in (0, \bar \psi)$, it proves that $\rho$ indeed converges to $ \sqrt{\bar \psi}$. 

If $\bar \psi=0$, we prove that  $\limsup  \rho \leq 0$ with the same method, and the non-negativity of $\psi$ ensures that $\liminf \rho \geq 0$. 

Let $\varphi \in \mathcal{C}^1_c\left(\R^d\right)$, and let us denote $\bar \rho:=\underset{t\to +\infty}{\lim}\rho(t)$.
We recall that if a differentiable function converges, then its derivative is either divergent or converges to $0$. Hence, since $t\mapsto \int_{\R^d}{\varphi(x)n(t,x)dx}$ converges (by hypothesis), and 
\begin{align*}
\frac{d}{dt}\int_{\R^d}{\varphi(x)n(t,x)dx}&=-\int_{\R^d}{\nabla \cdot \left( f(x)n(t,x) \right) \varphi(x)dx}\int_{\R^d}{(r(x)-\rho(t))\varphi(x)n(t,x)dx}\\
&=+\int_{\R^d}{f(x).\nabla \varphi(x) n(t,x)dx}+\int_{\R^d}{\left(r(x)- \rho(t)\right)\varphi(x)n(t,x)dx}\\
&\underset{t\to +\infty}{\longrightarrow}\int_{\R^d}{f(x).\nabla \varphi(x)+(r(x)-\bar \rho) \varphi(x) d\mu(x)}, 
\end{align*}
the equality
\begin{align}
\int_{\R^d}{f(x).\nabla \varphi(x)+(r(x)-\bar \rho) \varphi(x) d\mu(x)}=0
\label{weak rho}
\end{align}
holds for any $\varphi\in \mathcal{C}^1_c(\R^d)$. 

It remains to prove that $\mu(\R^d)=\bar \rho$. If $\bar \rho =0$, then the non-negativity of $n$ and the definition of $\rho$ lead to $\mu=0$. 
Let us now assume that $\bar \rho>0$, and let $\e>0$. Since $\mu$ is a finite measure, $r\in \mathcal{C}^0(\R^d)$, and owing to the definition of $\psi$ and $\bar \rho$, there exists $K\subset \R^d$ a compact set such that
\begin{itemize}
\item  $\mu(K)\geq \mu(\R^d)-\varepsilon$ 
\item $\int_K{r(x)d\mu(x)dx}\geq \int_{\R^d}{r(x)d\mu(x)dx}-\e= \bar{\rho}^2-\e$. 
\end{itemize}
Let $\varphi_K\in \mathcal{C}^1_c(\R^d)$ such that $\varphi_K\equiv 1$ on $K$, $0\leq \varphi\leq 1$ on $\R^d$.
Since $\nabla \varphi_K \equiv 0$ on $K$,  $$\bigg\lvert \int_{\R^d}{f(x).\nabla \varphi_K(x) d\mu (x)} \bigg\rvert \leq \lVert f. \nabla \varphi_K \rVert_{\infty} \mu(\R^d\backslash K)\leq \e  \lVert f. \nabla \varphi_K \rVert_{\infty}. $$
Moreover, according to the choice of $\varphi_K$
  $$\int_{\R^d}{r(x)\varphi_K(x)d\mu(x)}\in [\bar{\rho}^2-\e, \bar{\rho}^2],\quad 
\text{and} 
 \quad \int_{\R^d}{\varphi_K(x)d\mu(x)}\in [\mu(\R^d)-\e, \mu(\R^d)].$$
Hence, injecting these inequalities in \eqref{weak rho}, we obtain
\[-C\e\leq \bar \rho(\bar \rho-\mu(\R^d))\leq C\e\]
for some $C\geq 0$. Since this equality holds for any $\e$, and $\bar \rho$ is positive, it proves that $\mu(\R^d)=\bar \rho$. 
\end{proof}
Weak stationary solutions which are smooth enough (at least in $\mathcal{C}^1(\R^d)$) are in  fact stationary solutions in the strong sense, as defined in the following lemma, and can be further characterised. 
\begin{lem}
Let $\bn \in \mathcal{C}^1(\R^d)$. Then, $\bn$ is a weak stationary solution of \eqref{eq intro} if and only if for all $t\geq 0$, $y\in \R^d$, 
\begin{align*}
\begin{cases}
  \bn(X(t,y))=e^{\int_0^t{\tr(X(s,y))-\bar \rho \; ds}} \; \bn(y)\\
\int_{\R^d}{\bar n (x)dx}=\bar \rho
\end{cases}. 
\end{align*} \label{criterion stationary sol }
\end{lem}
\begin{proof}
First, let us note that, since $\bn \in \mathcal{C}^1(\R^d)$, one can integrate by parts in the expression \eqref{weak solution} in order to prove that $\bn$ is a weak stationary solution if and only if for any $\varphi\in \mathcal{C}^1_c(\R^d)$, 
\[\int_{\R^d}{\left( -\nabla \cdot \left(f(x) \bn(x)\right)+\left(r(x)-\bar \rho \right)\bn(x) \right)\varphi(x)dx}=0,\]
with $\bar \rho=\int_{\R^d}{\bn(x)dx}$, which means that $\bn$ is a weak stationary solution if and only if it is a stationary solution in the strong sense, \textit{i.e}
\[ -\nabla \cdot \left(f(x) \bn(x)\right)+\left(r(x)-\bar \rho \right)\bn(x)=0 \quad \text{ for all } x \in \R^d.\]
The result follows, since for any $y\in \R^d$ 
\begin{align*}
 \frac{d}{dt}\left(\bn(X(t,y))e^{-\int_0^t{\tr(X(s,y))-\bar \rho \; ds}}\right)&=\bigg(f(X(t,y)).\nabla \bn(X(t,y)) -\left( \tr(X(t,y)- \bar \rho \right) \bn(X(t,y))\bigg) e^{-\int_0^t{\tr(X(s,y))-\bar \rho \; ds}}\\
  =-\bigg(-\nabla \cdot &\left(f(X(t,y)) \bn(X(t,y))\right)+(r(X(t,y))-\bar \rho )\bn(X(t,y))\bigg)e^{-\int_0^t{\tr(X(s,y))-\bar \rho \; ds}}=0.
\end{align*}
\end{proof}
Lemma \ref{criterion stationary sol } allows us to conclude that in the case where the ODE $\dot u=f(u)$ has a repellor with a bounded unstable set, there exists a smooth stationary solution for~\eqref{eq intro}. 
\begin{cor}
Let $x_u\in \R^d$ be a repellor point for the ODE $\dot x =f(x)$, and let us assume that  
\[\bn(x):= \frac{\tr(x_u)}{\alpha} e^{\int_0^{+\infty}{\tr(Y(s,x))-\tr(x_u)ds}}\ind_{B}(x) \]
is well-defined, and that $\bar n \in \mathcal{C}^1(B)\cap L^1(B)$, where $B=\{x\in \R^d: Y(t,x)\ul x_u\}$ is the unstable set of~$x_u$.
Then, $\bar{n}$ is a $\mathcal{C}^1$ stationary solution. 
\end{cor}
\begin{proof}
%First, let us note that $\bn$ is well-defined and $\mathcal{C}^1$, since for any compact set $K\subset B$, there exist $C, \delta>0$ such that $\lVert \tr(Y(t,x))-\tr(x_u) \rVert\leq Ce^{-\delta t}$, for all $t\geq 0$. 
For all $y\in \R$, $t\geq 0$, 
\[\bar n(X(t,y))=\frac{\tr(x_u)}{\alpha} e^{\int_0^{+\infty}{\tr(X(t-s, x))-\tr(x_u)ds}}= \frac{\tr(x_u)}{\alpha} e^{\int_{-\infty}^{t}{\tr(X(s, y))-\tr(x_u)ds}},  \]
with the change of variable $s'=t-s$
and 
\[\bar n(y)=\frac{\tr(x_u)}{\alpha}e^{\int_{- \infty}^0{\tr(X(s,y))}-\tr(x_u)ds}, \]
with the change of variable $s'=-s$. 
Thus, the equality of Lemma \ref{criterion stationary sol } holds, which concludes the proof. 
\end{proof}

\bibliographystyle{plain}
\bibliography{biblio_sel_adv}

\end{document}